\documentclass[ejs,preprint]{imsart}

\RequirePackage[OT1]{fontenc}
\usepackage[utf8]{inputenc}
\RequirePackage{amsthm,amssymb,amsmath}
\RequirePackage[numbers]{natbib}
\RequirePackage[colorlinks,citecolor=blue,urlcolor=blue]{hyperref}
\usepackage{graphicx}

\usepackage{enumitem}

\usepackage{accents}
\usepackage{bbm}
\usepackage{subfig}

\usepackage{listofsymbols}
\usepackage{nomencl}
\makenomenclature
 
\startlocaldefs
\numberwithin{equation}{section}
\theoremstyle{plain}
\newtheorem{thm}{Theorem}[section]
\endlocaldefs

\theoremstyle{definition}
\newtheorem{defn}[thm]{\protect\definitionname}
\theoremstyle{plain}
\newtheorem{lem}[thm]{\protect\lemmaname}
\theoremstyle{plain}
\newtheorem{prop}[thm]{\protect\propositionname}
\theoremstyle{plain}

\theoremstyle{plain}
\newtheorem{cor}[thm]{\protect\corollaryname}
\theoremstyle{remark}
\newtheorem*{claim*}{\protect\claimname}
\theoremstyle{plain}
\newtheorem*{thm*}{\protect\theoremname}  
\theoremstyle{plain}
\newtheorem*{lem*}{\protect\lemmaname}
\theoremstyle{plain}
\newtheorem*{prop*}{\protect\propositionname}
\theoremstyle{remark}
\newtheorem{rem}[thm]{Remark}
\theoremstyle{remark}
\newtheorem*{rem*}{Remark}

\usepackage[english]{babel}
  \providecommand{\corollaryname}{Corollary}
  \providecommand{\definitionname}{Definition}
  \providecommand{\claimname}{Claim}
  \providecommand{\lemmaname}{Lemma}
  \providecommand{\propositionname}{Proposition}
\providecommand{\theoremname}{Theorem}

\newcommand{\R}{\mathbb{R}}
\renewcommand{\P}{\mathbb{P}}

\newcommand{\E}{\mathbb{E}}
\newcommand{\X}{\mathbb{X}}
\newcommand{\norm}[1]{\left\Vert #1\right\Vert}

\begin{document}

\begin{frontmatter}
\title{Estimating the Reach of a Manifold}
\runtitle{Estimating the Reach of a Manifold}

\begin{aug}
\author{\fnms{Eddie} \snm{Aamari}%
\thanksref{a1,t1,t2}%
\ead[label=e1]{aamari@lpsm.paris}%
\ead[label=u1,url]{lpsm.paris/~aamari/}%
}%
,
\author{\fnms{Jisu} \snm{Kim}%
\thanksref{a3,t4,t5}\ead[label=e2]{jisuk1@andrew.cmu.edu}%
\ead[label=u2,url]{stat.cmu.edu/~jisuk/}%
}%
,
\\
\author{\fnms{Fr\'ed\'eric} \snm{Chazal}%
\thanksref{a3,t1,t2}%
\ead[label=e3]{frederic.chazal@inria.fr}%
\ead[label=u3,url]{geometrica.saclay.inria.fr/team/Fred.Chazal/}%
}%
,
\author{\fnms{Bertrand} \snm{Michel}%
\thanksref{a4,t1,t2}%
\ead[label=e4]{bertrand.michel@ec-nantes.fr}%
\ead[label=u4,url]{bertrand.michel.perso.math.cnrs.fr/}%
}%
,
\\
\author{\fnms{Alessandro} \snm{Rinaldo}%
\thanksref{a2,t5}%
\ead[label=e5]{arinaldo@cmu.edu}%
\ead[label=u5,url]{stat.cmu.edu/~arinaldo/}%
}%
,
\and
\author{\fnms{Larry} \snm{Wasserman}%
\thanksref{a2}%
\ead[label=e6]{larry@stat.cmu.edu}%
\ead[label=u6,url]{stat.cmu.edu/~larry/}%
}

\thankstext{t1}{Research supported by ANR project TopData ANR-13-BS01-0008}
\thankstext{t2}{Research supported by Advanced Grant of the European Research Council GUDHI}
\thankstext{t4}{Supported by Samsung Scholarship}
\thankstext{t5}{Partially supported by NSF CAREER Grant DMS 1149677}
\runauthor{E. Aamari, J. Kim et al.}

\address[a1]{
CNRS, LPSM \\
Université Paris Diderot \\
\printead{e1}
\printead{u1}
}

\address[a3]{
Inria Saclay -- \^{I}le-de-France \\
\printead{e2,e3} \\
\printead{u2,u3}
}
\address[a2]{
Department of Statistics \\
Carnegie Mellon University \\
\printead{e5,e6} \\
\printead{u5,u6}
}

\address[a4]{
Département Informatique et Mathématiques \\
\'Ecole Centrale de Nantes \\
\printead{e4}
\printead{u4}
}

\end{aug}

\begin{abstract}
Various problems in manifold estimation make use of a quantity called the {\em reach}, denoted by $\tau_M$, which is a measure of the regularity of the manifold. This paper is the first investigation into the problem of how to estimate the reach. First, we study the geometry of the reach through an approximation perspective. 
We derive new geometric results on the reach for submanifolds without boundary. An estimator $\hat{\tau}$ of $\tau_M$ is proposed in an oracle framework where tangent spaces are known, and bounds assessing its efficiency are derived. In the case of i.i.d. random point cloud $\X_n$, $\hat{\tau}(\X_n)$ is showed to achieve uniform expected loss bounds over a $\mathcal{C}^3$-like model. Finally, we obtain upper and lower bounds on the minimax rate for estimating the reach.

\end{abstract}



\end{frontmatter}

\section{Introduction}
\label{sec:introduction}

\subsection{Background and Related Work}

Manifold estimation
has become an increasingly important problem in statistics and machine learning.
There is now a large literature on methods and theory for
estimating manifolds.
See, for example,
\cite{Kim2015,Genovese12,Fefferman16,
Boissonnat14,
NiyogiSW2008,
belkin2006manifold,gine2006empirical}.

Estimating a manifold, or functionals of a manifold,
requires regularity conditions.
In nonparametric function estimation, regularity conditions often
take the form of smoothness constraints.
In manifold estimation problems,
a common assumption is that the reach $\tau_M$ of the manifold $M$ is non-zero.

First introduced by Federer \cite{Federer1959}, 
the reach $\tau_M$ of a set $M\subset \R^D$ is the largest number such that any point at distance less than $\tau_M$ from $M$ has a unique nearest point on $M$.
If a set has its reach greater than $\tau_{min}>0$, then one can roll freely a ball of radius $\tau_{min}$ around it \cite{Cuevas12}. 
The reach is affected by two factors: the curvature of the manifold and
the width of the narrowest bottleneck-like structure of $M$, which quantifies how
close $M$ is from being self-intersecting.

Positive reach is the minimal regularity assumption on sets in geometric measure
theory and integral geometry \cite{Federer69,Thale08}. 
Sets with positive reach exhibit a structure that is close to being differential --- the
so-called tangent and normal cones. The value of the reach itself quantifies the
degree of regularity of a set, with larger values associated to more regular
sets.
The positive reach assumption is routinely imposed in the statistical analysis of geometric structures in order to ensure good statistical properties \cite{Cuevas12} and to derive
theoretical guarantees. 
For example, in manifold reconstruction, the reach helps formalize minimax rates~\cite{Genovese12,Kim2015}. 
The optimal manifold estimators of \cite{Aamari15} implicitly use reach as a scale parameter in their construction.
In homology inference \cite{NiyogiSW2008,Balakrishnan12}, the reach drives the minimal sample size required to consistently estimate topological invariants.
It is used in \cite{Cuevas07} as a regularity parameter in the estimation of the Minkowski boundary lengths and surface areas.
The reach has also been explicitly used as a regularity parameter in geometric inference, such as in volume estimation \cite{Arias16} and manifold clustering~\cite{Arias13}.
Finally, the reach often plays the role of a scale parameter in dimension reduction techniques such as vector
diffusions maps \cite{Singer12}. 
Problems in computational geometry such as manifold reconstruction also rely on
assumptions on the
reach \cite{Boissonnat14}.

In this paper we study the problem of estimating reach.
To do so, we first provide new geometric results on the reach.
We also give the first bounds on the minimax rate
for estimating reach.
As a first attempt to study reach estimation in the literature, we will mainly work in a framework where a point cloud is observed jointly with tangent spaces, before relaxing this constraint in Section~\ref{sec:unknownTangent}.
Such an oracle framework has direct applications in digital imaging~\cite{Klette04,Hug17}, where a very high resolution image or 3D-scan, represented as a manifold, enables to determine precisely tangent spaces for arbitrary finite set of points \cite{Hug17}.

There are very few papers on this problem.
When the embedding dimension is $3$, the estimation of the local feature size (a localized
version of the reach) was tackled in a deterministic way in
\cite{Dey06}. To some extent, the estimation of the medial axis 
(the set of points that have strictly more than one nearest point
on $M$) and its generalizations~\cite{Cuevas14,Attali09} can be viewed
as an indirect way to estimate the reach.
A test procedure designed to
validate whether data actually comes from a smooth manifold satisfying a condition on the reach was
developed in~\cite{Fefferman16}. The authors derived a consistent test procedure, but the results do not permit any
inference bound on the reach.
When a sample is uniformly distributed over a full-dimensional set, \cite{Rodriguez16} proposes a selection procedure for the radius of $r$-convexity of the set, a quantity closely related to the reach.

\subsection{Outline}

In Section \ref{sec:framework} we provide some differential geometric background
and define the statistical problem at hand.  New
geometric properties of the reach are derived in Section
\ref{sec:geometry}, and their consequences for its inference follow {in}
Section \ref{sec:estimator} in a setting where tangent spaces are
known.  We then derive minimax bounds in
Section~\ref{sec:minimax}.
An extension to a model
where tangent spaces are unknown is discussed {in} Section
\ref{sec:unknownTangent}, and we conclude with some open questions in
Section \ref{sec:conclusion}.
For sake of readability, the proofs are given in the Appendix.

\section{Framework}
\label{sec:framework}

\subsection{Notions of Differential Geometry}
In what follows, $D\geq 2$ and $\mathbb{R}^D$ is endowed with the Euclidean
inner product $\left\langle \cdot, \cdot \right\rangle$ and the associated norm $\norm{\cdot}$.
The associated closed ball of radius $r$ and center $x$ is denoted by $\mathcal{B}(x,r)$.
We will consider compact connected submanifolds $M$  of $\R^D$ of fixed
dimension $1\leq d < D$ and without boundary \cite{DoCarmo92}. 
For every point $p$ in $M$, the tangent space of $M$ at $p$ is denoted
by $T_p M$: 
it is the $d$-dimensional linear subspace of $\R^D$ composed of the directions that $M$ spans in the neighborhood of $p$.
Besides the Euclidean structure given
by $\R^D \supset M$, a submanifold is endowed with an intrinsic distance
induced by the ambient Euclidean one, and called the geodesic distance.
Given a
smooth path $c: [a,b] \rightarrow M$, the length of $c$ is defined as $Length(c)
= \int_a^b \norm{c'(t)}dt$. One can show \cite{DoCarmo92} that there exists a path $\gamma_{p \rightarrow q}$ of minimal length joining $p$ and $q$. 
Such an arc is called geodesic, and the geodesic distance between $p$ and $q$ is given by $d_M(p,q) = Length(\gamma_{p \rightarrow q})$. 
We let $\mathcal{B}_M(p,s)$ denote the closed geodesic ball of center $p\in M$ and of radius $s$.
A geodesic $\gamma$ such that $\norm{\gamma'(t)} = 1$ for all $t$ is called arc-length parametrized. 
Unless stated otherwise, 
we always assume that geodesics are parametrized by arc-length.
For all $p \in M$ and all unit vectors $v \in T_p M$, we denote by $\gamma_{p,v}$ the unique arc-length parametrized geodesic of $M$ such that $\gamma_{p,v}(0)=p$ and $\gamma'_{p,v}(0)=v$. The exponential map is defined as $\exp_p(vt) = \gamma_{p,v}(t)$.
 Note that from the compactness of $M$, $\exp_p: T_p M \rightarrow M$ is defined globally on $T_p M$. For any two nonzero vectors $u,v\in \mathbb{R}^{D}$, we let $\angle(u,v)=d_{\mathcal{S}^{D-1}}(\frac{u}{\left\Vert u\right\Vert},\frac{v}{\left\Vert v\right\Vert})$ be the angle between $u$ and $v$. 

\subsection{Reach}

First introduced by Federer \cite{Federer1959}, the reach regularity parameter
is defined as follows.
Given a closed subset $A \subset \R^D$, the medial axis $Med(A)$ of $A$ is the
subset of $\R^D$ consisting of the points that have at least two nearest neighbors on $A$. 
Namely, denoting by $d(z,A) = \inf_{p \in A} \norm{p-z}$ the distance function to $A$,
\begin{equation}
\label{eq:medialaxis}
Med(A)
=
\left\{
z \in \R^D
|
\exists p\neq q \in A, \norm{p-z} = \norm{q-z} = d(z,A)
\right\}.
\end{equation}
The reach of $A$ is then defined as the minimal distance from $A$ to $Med(A)$.
\begin{defn}
The reach of a closed subset $A \subset \R^D$ is defined as
\begin{align}\label{eqn:reach_with_medial_axis}
\tau_A 
=
\inf_{p \in A} d\left(p,Med(A) \right)
=
\inf_{z \in Med(A)} d\left(z,A \right).
\end{align}
\end{defn}
Some authors refer to $\tau_A^{-1}$ as the \textit{condition number} \cite{NiyogiSW2008,Singer12}.
From the definition of the medial axis in \eqref{eq:medialaxis}, the projection $\pi_A(x) = \arg\min_{p \in A} \norm{p-x}$ onto $A$ is well defined outside $Med(A)$. 
The reach is the largest distance $\rho \geq 0$ such that $\pi_A$ is well defined on the $\rho$-offset $\left\{ x \in \R^D | d(x,A) < \rho \right\}$. 
Hence, the reach condition can be seen as a generalization of convexity, since a set $A \subset \R^D$ is convex if and only if $\tau_A = \infty$.
In the case of submanifolds, one can reformulate the definition of the reach in the following manner.

\begin{thm}[Theorem 4.18 in \cite{Federer1959}]
\begin{align}\label{eqn:reach_as_a_supremum_federer}
\tau_M = \underset{q\neq p \in M}{\inf} ~ \frac{\norm{q-p}^2}{2d(q-p,T_p M)}.
\end{align}
\end{thm}

\begin{figure}[ht]
\centering
\includegraphics[clip = true, trim = 0 0 0 0.15\textwidth,width=0.8\textwidth]{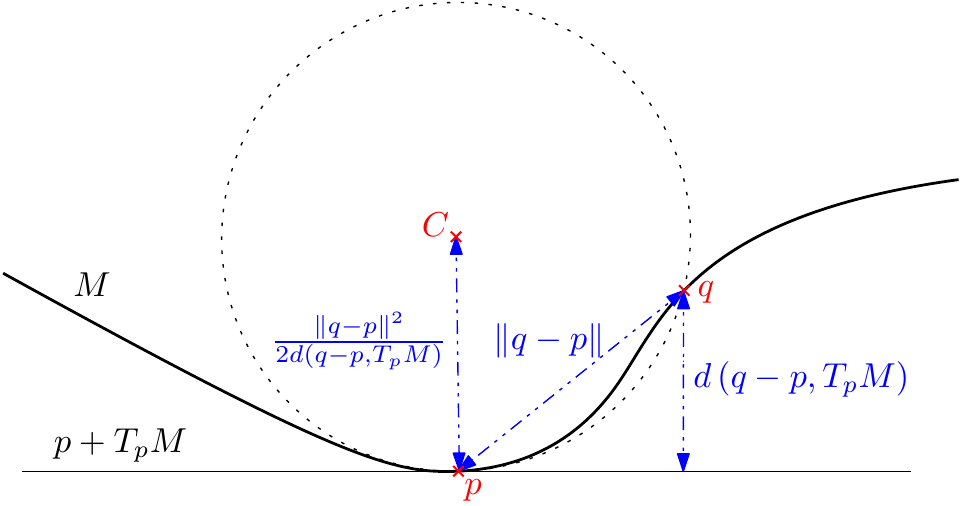}
\caption{Geometric interpretation of quantities involved in (\ref{eqn:reach_as_a_supremum_federer}).}\label{fig:tangent_ball_2d}
\end{figure}
This formulation has the advantage of involving only points on $M$ and its tangent spaces, while (\ref{eqn:reach_with_medial_axis}) uses the distance to the medial axis $Med(M)$, which is a global quantity. 
The formula \eqref{eqn:reach_as_a_supremum_federer} will be the starting point of the estimator proposed in this paper (see Section \ref{sec:estimator}).

The ratio appearing in~\eqref{eqn:reach_as_a_supremum_federer} can be interpreted geometrically, as suggested in Figure \ref{fig:tangent_ball_2d}. This ratio is the radius of an ambient ball, tangent to $M$ at $p$ and passing through $q$.
Hence, at a differential level, the reach gives a lower bound on the radii of curvature of $M$.
Equivalently, $\tau_M^{-1}$ is an upper bound on the curvature of $M$.

\begin{prop}[Proposition 6.1 in \cite{NiyogiSW2008}]
\label{thm:geometry:upper_bound_on_second_derivative}
Let $M \subset \R^D$ be a submanifold, and $\gamma_{p,v}$ an arc-length parametrized geodesic of $M$. Then for all $t$,
\[
\norm{\gamma_{p,v}''(t)} \leq 1/ \tau_{M}.
\]
\end{prop}
In analogy with function spaces, the class 
$\left\{M \subset \R^D | \tau_M \geq \tau_{min} > 0 \right\}$ 
can be interpreted as the Hölder space $\mathcal{C}^2(1/\tau_{min})$.
In addition, as illustrated in Figure~\ref{fig:bone_medial_axis}, the condition
$\tau_M \geq \tau_{min}> 0$ also prevents bottleneck structures where $M$ is
nearly self-intersecting. 
This idea will be made rigorous in Section~\ref{sec:geometry}.

\begin{figure}[ht]
\centering
\includegraphics[width=0.7\textwidth]{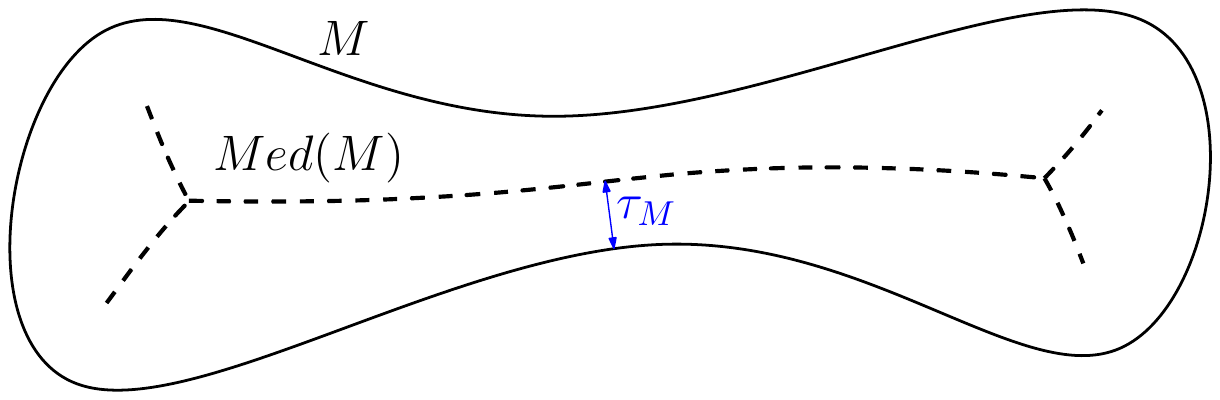}
\caption{
A narrow bottleneck structure yields a small reach $\tau_M$.}\label{fig:bone_medial_axis}
\end{figure}

\subsection{Statistical Model and Loss}
Let us now describe the regularity assumptions we will use throughout. To avoid
arbitrarily irregular shapes, we consider submanifolds $M$ with their reach lower bounded by $\tau_{min}>0$.
Since the parameter of interest $\tau_M$ is a $\mathcal{C}^2$-like quantity, it is natural --- and actually necessary, as we shall see in Proposition \ref{thm:minimax_nonconsistency} --- 
to require an extra degree of smoothness.
For example, by imposing an upper bound on the third order derivatives of geodesics.

\begin{defn}\label{def:geometric_model}
Let $\mathcal{M}^{d,D}_{\tau_{min},L}$ denote the set of compact connected $d$-dimensional submanifolds $M \subset \R^D$ without boundary such that $\tau_M \geq \tau_{min}$, and for which every arc-length parametrized geodesic $\gamma_{p,v}$ is $\mathcal{C}^3$ and satisfies
\begin{align*}
\norm{\gamma_{p,v}'''(0)} \leq L.
\end{align*}
\end{defn}
The regularity bounds $\tau_{min}$ and $L$ are assumed to exist only for the purpose of deriving uniform estimation bounds. 
However, we emphasize the fact that the forthcoming estimator $\hat{\tau}$ \eqref{eq:estimator.estimator} does not require them in its construction.

It is important to note that any compact $d$-dimensional $\mathcal{C}^3$-submanifold $M \subset \R^D$ belongs to such a class $\mathcal{M}^{d,D}_{\tau_{min},L}$, provided that $\tau_{min} \leq \tau_M$ and that $L$ is large enough.
Note also that since the third order condition $\norm{\gamma_{p,v}'''(0)} \leq L$ needs to hold for all $(p,v)$, we have in particular that $\norm{\gamma_{p,v}'''(t)} \leq L$ for all $t \in \R$.
To our knowledge, such a quantitative $\mathcal{C}^3$  assumption on the geodesic trajectories has not been considered in the computational geometry literature. 

Any submanifold $M \subset \R^D$ of dimension $d$ inherits a natural measure $vol_M$ from the $d$-dimensional Hausdorff measure $\mathcal{H}^d$ on $\R^D$ \cite[p. 171]{Federer69}.
We will consider distributions $Q$ that have densities with respect to $vol_M$ that are bounded away from zero.

\begin{defn}\label{def:model_with_unknown_tangent_spaces}
We let $\mathcal{Q}^{d,D}_{\tau_{\min},L,f_{\min}}$ denote the set of
distributions $Q$ having support $M \in \mathcal{M}^{d,D}_{\tau_{min},L}$ and
with a Hausdorff density $f = \frac{d Q}{d vol_M}$ satisfying $\inf_{x \in
M}f(x) \geq f_{\min} >0$ on $M$.
\end{defn}
As for $\tau_{min}$ and $L$, the knowledge of $f_{min}$ will not be required in the construction of the estimator $\hat{\tau}$ \eqref{eq:estimator.estimator} described below.

In order to focus on the geometric aspects of the reach, we will first consider the case where tangent spaces are observed at all the sample points. 
As mentioned in the introduction, the knowledge of tangent spaces is a reasonable assumption in digital imaging~\cite{Klette04}. This assumption will eventually be relaxed in Section \ref{sec:unknownTangent}.

We let $\mathbb{G}^{d,D}$ denote the Grassmannian of dimension $d$ of $\R^D$, that is the set of all $d$-dimensional linear subspaces of $\R^D$. 

\begin{defn}\label{def:model_with_known_tangent_spaces}
For any distribution $Q \in \mathcal{Q}^{d,D}_{\tau_{\min},L,f_{\min}}$ with support $M$ we associate the distribution $P$ of the random variable $(X,T_X M)$ on $\R^D \times \mathbb{G}^{d,D}$, where $X$ has distribution $Q$. We let $\mathcal{P}^{d,D}_{\tau_{\min},L,f_{\min}}$ denote the set of all such distributions.
\end{defn}

Formally, one can write $P(d x \, d T) = \delta_{T_x M}(d T) Q(d x)$, where $\delta_{\cdot}$ denotes the Dirac measure.
An i.i.d. $n$-sample of $P$ is of the form $(X_1,T_1),\ldots,(X_n,T_n) \in \R^D \times \mathbb{G}^{d,D}$, where $X_1,\ldots,X_n$ is an i.i.d. $n$-sample of $Q$ and $T_i = T_{X_i} M$ with $M = supp(Q)$.
For a distribution $Q$ with support $M$ and associated distribution $P$ on $\R^D
\times \mathbb{G}^{d,D}$, we will write $\tau_P = \tau_Q = \tau_M$, with a slight abuse of notation.

Note that the model does not explicitly impose an upper bound on $\tau_M$. 
Such an upper bound would be redundant, since the lower bound on $f_{min}$ does impose such an upper bound, as we now state in the following result.
The proof relies on a volume argument (Lemma
\ref{thm:geometry:diameter_bound}), leading to a bound on the diameter of $M$,
and on a topological argument (Lemma \ref{thm:geometry:reach_vs_diameter}) to link the reach and the diameter.

\begin{prop}\label{thm:geometry:max_reach_volumic}
Let $M\subset \R^D$ be a connected closed $d$-dimensional manifold, and let $Q$
be a probability distribution with support $M$. Assume that $Q$ has a density $f$  with
respect to the Hausdorff measure on $M$ such that  $\inf_{x \in M} f(x)  \geq f_{min}>0$. Then,
\[
\tau_M^d \leq \frac{C_d}{f_{min}},
\]
for some constant $C_d>0$ depending only on $d$.
\end{prop}

To simplify the statements and the proofs, we focus on a loss involving the condition number. Namely, we measure the error with the loss
\begin{equation}\label{eq:loss}
\ell(\tau,\tau') = \left| \frac{1}{\tau} - \frac{1}{\tau'} \right|^p, \quad p\geq 1.
\end{equation}

In other words, we will consider the estimation of the condition number $\tau_M^{-1}$ instead of the reach $\tau_M$.

\begin{rem}\label{rmk:equivalence_of_reach_and_condition_number}
For a distribution $P \in \mathcal{P}^{d,D}_{\tau_{\min},L,f_{\min}}$, Proposition \ref{thm:geometry:max_reach_volumic} asserts that $\tau_{min} \leq \tau_P \leq \tau_{max} := \left(C_d/f_{min}\right)^{1/d}$. Therefore, in an inference set-up, we can always restrict to estimators $\hat{\tau}$ within the bounds $\tau_{min} \leq \hat{\tau} \leq \tau_{max}$. Consequently,
\begin{equation*}
\frac{1}{\tau_{max}^{2p}} \left| \tau_P - \hat{\tau} \right|^p
\leq
\left| \frac{1}{\tau_P} - \frac{1}{\hat{\tau}} \right|^p
\leq
\frac{1}{\tau_{min}^{2p}} \left| \tau_P - \hat{\tau} \right|^p,
\end{equation*}
so that the estimation of the reach $\tau_P$ is equivalent to the estimation of the condition number $\tau_P^{-1}$, up to constants. 
\end{rem}

With the statistical framework developed above, we can now see explicitly why
the third order condition $\norm{\gamma'''} \leq L < \infty$ is necessary.
Indeed, the following Proposition \ref{thm:minimax_nonconsistency} demonstrates that relaxing this constraint ---
\textit{i.e.} setting $L = \infty$ --- renders the problem of reach estimation
intractable.
Its proof is to be found in Section \ref{sec:appendix:construction}.
Below, $\sigma_d$ stands for the volume of the $d$-dimensional unit sphere $\mathcal{S}^{d}$.

\begin{prop}\label{thm:minimax_nonconsistency}
There exists a universal constant $c > 1/100$ such that given $\tau_{min} > 0$, provided that $f_{min} \leq ({2^{d+1}\tau_{min}^d \sigma_d})^{-1}$, we have for all $n\geq1$,
\[
\inf_{\hat{\tau}_n} \sup_{P \in \mathcal{P}^{d,D}_{\tau_{min},L=\infty, f_{min}}} \E_{P^n} \left| \frac{1}{\tau_P} - \frac{1}{\hat{\tau}_n} \right|^p \geq \left(\frac{c}{\tau_{min}}\right)^p,
\]
where the infimum is taken over the estimators $\hat{\tau}_n = \hat{\tau}_n\left(X_1,T_1,\ldots,X_n,T_n \right)$. 
\end{prop}

Thus, one cannot expect to derive consistent uniform approximation bounds  for the reach solely under the condition $\tau_M \geq \tau_{min}$. This result is natural, since the problem at stake is to estimate 
a differential quantity of order two. Therefore, some notion of uniform $\mathcal{C}^3$ regularity is needed.

\section{Geometry of the Reach}

\label{sec:geometry}

In this section, we give a precise geometric description of how the reach arises.
In particular, below we will show that the reach is determined either by a bottleneck structure or an area of high curvature (Theorem \ref{thm:geometry.local_global}). 
These two cases are referred to as \textit{global} reach and \textit{local}
reach, respectively.
All the proofs for this section are to be found in Section \ref{sec:appendix:geometry_of_the_reach}.

Consider the formulation (\ref{eqn:reach_with_medial_axis}) of the reach as the infimum of the distance between $M$ and its medial axis $Med(M)$.
By definition of the medial axis \eqref{eq:medialaxis}, if the infimum is attained it corresponds to a point $z_0$ in $Med(M)$ at distance $\tau_M$ from $M$, which we call an \textit{axis point}.
Since $z_0$ belongs to the medial axis of $M$, it has at least two nearest neighbors $q_1,q_2$ on $M$, which we call a \textit{reach attaining pair} 
(see Figure \ref{subfig:trichotomy.reachpair}).
By definition, $q_1$ and $q_2$ belong to $\mathcal{B}(z_0,\tau_M)$ and cannot be farther than $2\tau_M$ from each other. 
We say that $(q_1,q_2)$ is a \textit{bottleneck} of $M$ in the extremal case $\norm{q_2-q_1} = 2\tau_M$ of antipodal points of $\mathcal{B}(z_0,\tau_M)$
(see Figure \ref{subfig:trichotomy.bottleneck}).
Note that the ball $\mathcal{B}(z_0,\tau_M)$ meets $M$ only on its boundary $\partial \mathcal{B}(z_0,\tau_M)$.

\begin{defn}	\label{def:geometry.bottleneck}
	Let $M \subset \R^D$ be a submanifold with reach $\tau_{M}>0$.
\begin{itemize}
	\item A pair of points $(q_{1},q_{2})$ in $M$ is called \textit{reach attaining} if there exists $z_{0}\in Med(M)$ such that $q_{1},q_{2}\in{\mathcal{B}(z_{0},\tau_{M})}$. 
	We call $z_{0}$ the \textit{axis point} of $(q_{1},q_{2})$,
	and $\norm{q_{1}-q_{2}}\in(0,2\tau_M]$ its \textit{size}.
	\item A reach attaining pair $(q_{1},q_{2})\in M^{2}$ is said to be a \textit{bottleneck} of
	$M$ if its size is $2\tau_{M}$,
	that is $\norm{q_{1}-q_{2}}=2\tau_{M}$. 
\end{itemize}
\end{defn}

As stated in the following Lemma \ref{thm:geometry.circlearc}, if a reach attaining pair is not a bottleneck --- that is $\norm{q_1-q_2} < 2\tau_M${, as in  Figure \ref{subfig:trichotomy.reachpair} }---, then $M$ contains an arc of a circle of radius $\tau_M$. 
In this sense, this ``semi-local" case --- when $\norm{q_1-q_2}$ can be arbitrarily small --- is not generic. 
Though, we do not exclude this case in the analysis.
\begin{lem}
	\label{thm:geometry.circlearc}
Let $M\subset \mathbb{R}^D$ be a compact submanifold with reach $\tau_{M}>0$. 
Assume that $M$ has a reach attaining pair $(q_{1},q_{2})\in M^{2}$ with size $\left\Vert q_{1}-q_{2} \right\Vert < 2\tau_{M}$. 
Let $z_{0} \in Med(M)$ be their associated axis point, and write $c_{z_{0}}(q_{1},q_{2})$ for the shorter arc of the circle with center $z_{0}$ and endpoints as $q_{1}$ and $q_{2}$. 

Then $c_{z_{0}}(q_{1},q_{2})\subset M$, and this arc (which has constant curvature $1/{\tau_{M}}$) is the geodesic joining $q_{1}$ and $q_{2}$.
\end{lem}
In particular, in this ``semi-local" situation,  since $\tau_M^{-1}$ is the norm of the second derivative of a geodesic of $M$ 
(the exhibited shorter arc of the circle of radius $\tau_M$), the reach can be viewed as
arising from directional curvature.

Now consider the case where the infimum \eqref{eqn:reach_with_medial_axis} is not attained.
In this case, the following Lemma \ref{thm:geometry.principalcurvature} asserts that $\tau_M$ is created by curvature. 
\begin{lem}\label{thm:geometry.principalcurvature}
Let $M \subset \R^D$ be a compact submanifold with reach $\tau_{M}>0$.
Assume that for all $z\in Med(M)$, $d(z,M) > \tau_{M}$. 
Then there exists $q_{0}\in M$ and a geodesic $\gamma_{0}$ such that $\gamma_{0}(0)=q_{0}$ and $\left\Vert \gamma_{0}''(0) \right\Vert =\frac{1}{\tau_{M}}$.
\end{lem}

To summarize, there are three distinct geometric instances in which the reach
may be realized:

\begin{itemize}
\item (See Figure \ref{subfig:trichotomy.bottleneck})
$M$ has a bottleneck: by definition, $\tau_M$ originates from a structure having scale $2\tau_M$.
\item (See Figure \ref{subfig:trichotomy.reachpair})
$M$ has a reach attaining pair but no bottleneck: then $M$ contains an arc of a circle of radius $\tau_M$ (Lemma \ref{thm:geometry.circlearc}), so that $M$ actually contains a zone with radius of curvature $\tau_M$.		
\item (See Figure~\ref{subfig:trichotomy.reachpoint})
$M$ does not have a reach attaining pair: then $\tau_M$ comes from a curvature-attaining point (Lemma~\ref{thm:geometry.principalcurvature}), that is
a point with radius of curvature~$\tau_M$.
\end{itemize}

\begin{figure}
\centering
		\subfloat[A bottleneck.]{
			\includegraphics[width=0.5\textwidth]{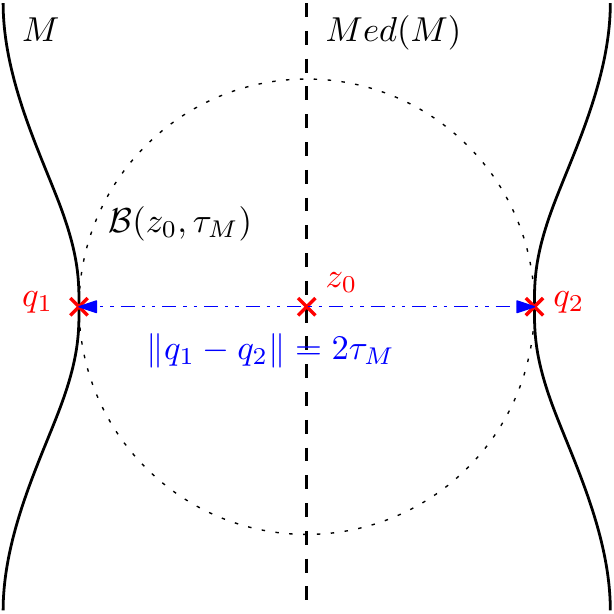}
			\label{subfig:trichotomy.bottleneck}
		}

		\subfloat[
		A non-bottleneck reach attaining pair.
		]{		
			\includegraphics[width=0.5\textwidth]{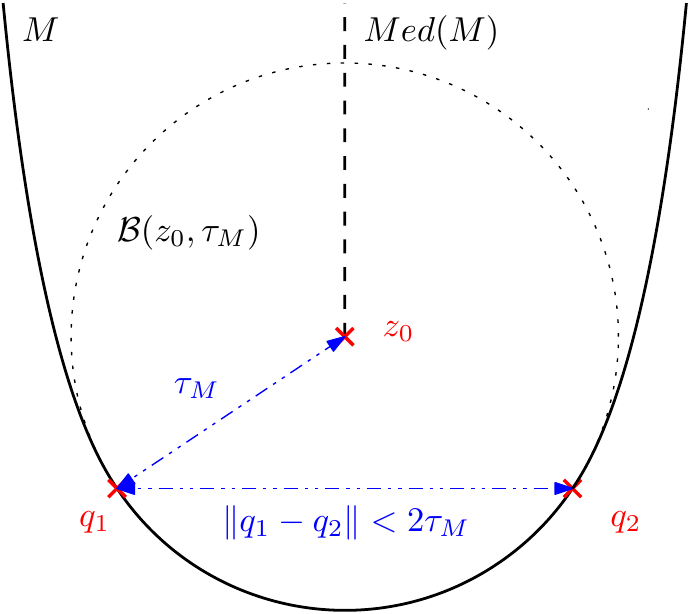}
			\label{subfig:trichotomy.reachpair}
		}
		\subfloat[
		{
		Curvature-attaining point.
		}	
		]{
			\includegraphics[width=0.5\textwidth]{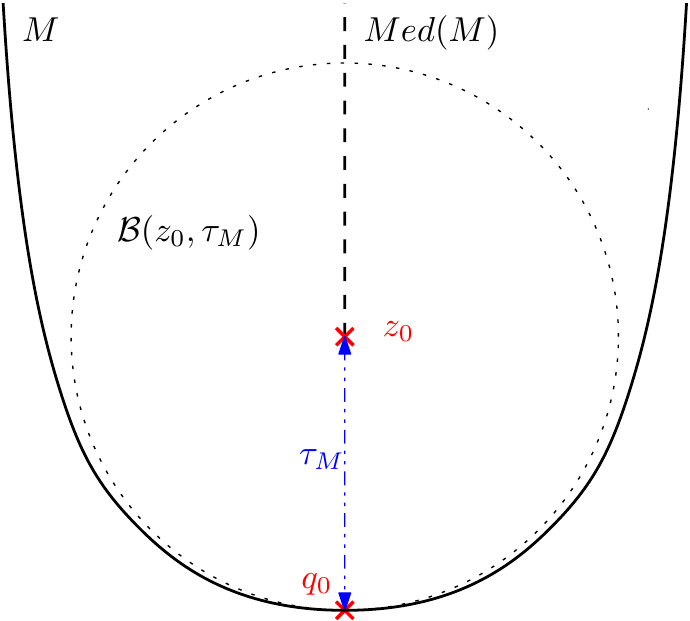}
			\label{subfig:trichotomy.reachpoint}
		}

	\caption{
		The different ways for the reach to be attained, as described in Lemma~\ref{thm:geometry.circlearc} and Lemma~\ref{thm:geometry.principalcurvature}.
	}
	\label{fig:geometry.trichotomy}
\end{figure}

From now on, we will treat the first case separately from the other two. We are now in a position to state the main result of this section. It is a straightforward consequence of Lemma~\ref{thm:geometry.circlearc} and Lemma \ref{thm:geometry.principalcurvature}.

\begin{thm}\label{thm:geometry.local_global}
	Let $M\subset \mathbb{R}^D$ be a compact submanifold with reach $\tau_{M}>0$. At least one of the following two assertions holds.
	\begin{itemize}
		\item\emph{(Global Case)} $M$ has a bottleneck $(q_{1},q_{2})\in M^{2}$, that is, there exists $z_{0}\in Med(M)$ such that $q_{1},q_{2}\in\partial\mathcal{B}(z_{0},\tau_{M})$ and $\left\Vert q_{1}-q_{2}\right\Vert =2\tau_{M}$. 
		\item\emph{(Local Case)} There exists $q_{0}\in M$ and an arc-length parametrized geodesic $\gamma_{0}$ such that $\gamma_{0}(0)=q_{0}$ and $\left\Vert \gamma_{0}''(0)\right\Vert =\frac{1}{\tau_{M}}$.
	\end{itemize}
\end{thm}
Let us emphasize the fact that the global case and the local case of Theorem~\ref{thm:geometry.local_global} are not mutually exclusive.
Theorem \ref{thm:geometry.local_global} provides a description of the
reach as arising from global and local geometric structures that, to the best of our knowledge, is new.
Such a distinction is especially important in our problem. 
Indeed, the global and local cases may yield different approximation properties and
require different statistical analyses. 
However, since one does not know a priori whether the reach arises from a global or a
local structure, an estimator of $\tau_M$ should be able to handle both cases simultaneously.

\section{Reach Estimator and its Analysis}\label{sec:estimator}
In this section, we propose an estimator $\hat{\tau}(\cdot)$ for the reach and
demonstrate its properties and rate of consistency under the loss \eqref{eq:loss}. 
For the sake of clarity in the analysis, we assume the tangent spaces to be known at every sample point. 
This assumption will be relaxed in Section \ref{sec:unknownTangent}.

We rely on the formulation of the reach given in
(\ref{eqn:reach_as_a_supremum_federer}) (see also Figure
\ref{fig:tangent_ball_2d}), and define $\hat{\tau}$ as a plugin estimator as
follows:
given a point cloud $\X \subset M$,
\begin{equation}\label{eq:estimator.estimator}
\hat{\tau}(\mathbb{X})=\inf_{x \neq y \in \X}\frac{\norm{y-x}^{2}}{2d(y-x,T_{x}M)}.
\end{equation}
In particular, we have $\hat{\tau}(M) = \tau_M$. 
Since the infimum \eqref{eq:estimator.estimator} is taken over a set $\X$ smaller than $M$, $\hat{\tau}(\X)$ always overestimates $\tau_M$. 
In fact, $\hat{\tau}(\X)$  is decreasing in the number of distinct
points in $\X$, a useful property that we formalize in the following result,
whose proof is immediate. 
\begin{cor}	\label{thm:estimator.overestimate}
Let $M$ be a submanifold with reach $\tau_{M}$ and $\mathbb{Y} \subset \mathbb{X}\subset M$ be two nested subsets. Then $\hat{\tau}(\mathbb{Y})\geq\hat{\tau}(\mathbb{X})\geq\tau_{M}.$

\end{cor}

We now derive the rate of convergence of $\hat{\tau}$.  
We analyze the global case (Section~\ref{subsec:estimator.global}) and the local
case (Section \ref{subsec:estimator.local}) separately.
In both cases, we first determine the performance of the estimator in a
deterministic framework, and then derive an expected loss bounds when
$\hat{\tau}$ is applied to a random sample.

Respectively, the proofs for Section \ref{subsec:estimator.global} and Section \ref{subsec:estimator.local} are to be found in Section \ref{sec:appendix:global_case} and Section \ref{sec:appendix:local_case}.

\subsection{Global Case}\label{subsec:estimator.global}

Consider the global case, that is, $M$ has a bottleneck structure (Theorem \ref{thm:geometry.local_global}).
Then the infimum (\ref{eqn:reach_as_a_supremum_federer}) is achieved at a bottleneck pair $(q_1,q_2) \in M^2$. 
When $\X$ contains points that are close to $q_1$ and $q_2$, one may expect that
the infimum over the sample points should also be close to
\eqref{eqn:reach_as_a_supremum_federer}:  
that is, that $\hat{\tau}(\X)$ should be close to $\tau_M$.

\begin{prop}\label{thm:estimator.global_twopoints}
Let $M \subset \mathbb{R}^D$ be a submanifold with reach $\tau_{M}>0$ that has a bottleneck $(q_{1},q_{2})\in M^{2}$ (see Definition \ref{def:geometry.bottleneck}), and $\mathbb{X}\subset M$. 
If there exist $x,y\in \X$ with $\left\Vert q_{1}-x \right\Vert <\tau_{M}$ and $\left\Vert q_{2}-y \right\Vert <\tau_{M}$, then
	\[
	0
	\leq
	\frac{1}{\tau_{M}}-\frac{1}{\hat{\tau}(\mathbb{X})}
	\leq
	\frac{1}{\tau_{M}}-\frac{1}{\hat{\tau}(\{x,y\})}
	\leq
	\frac{4}{\tau_{M}^{2}}\max\left\{ d_M(q_{1},x),d_M(q_{2},y)\right\}.
	\]
\end{prop}
The error made by $\hat{\tau}(\X)$ decreases linearly in the maximum of the distances to the critical points $q_1$ and $q_2$. 
In other words, the radius of the tangent sphere in Figure
\ref{fig:tangent_ball_2d} grows at most linearly in $t$ when we perturb by  $t< \tau_M$
its basis point $p=q_1$ and the point $q=q_2$ it passes through.

Based on the deterministic bound in Proposition
\ref{thm:estimator.global_twopoints}, we can now give an upper bound on the
expected loss under the model $\mathcal{P}_{\tau_{\min},L,f_{\min}}^{d,D}$. 
We recall that, throughout the paper, $\X_n = \left\{X_1,\ldots,X_n\right\}$ is an i.i.d. sample with common distribution $Q$ associated to $P$ (see Definition \ref{def:model_with_known_tangent_spaces}).

\begin{prop}\label{thm:estimator.global_risk}
	Let $P\in\mathcal{P}_{\tau_{\min},L,f_{\min}}^{d,D}$ and $M=supp(P)$. Assume that $M$ has a bottleneck $(q_{1},q_{2})\in M^{2}$ (see Definition \ref{def:geometry.bottleneck}). Then,
	\[
	\mathbb{E}_{P^n}
	\left[
		\left|\frac{1}{\tau_{M}}-\frac{1}{\hat{\tau}(\X_n)}\right|^{p}
	\right]
	\leq 
	C_{\tau_{M},f_{\min},d,p}
	n^{-\frac{p}{d}},
	\]
	where $C_{\tau_{M},f_{\min},d,p}$ depends only on $\tau_{M}$, $f_{\min}$, $d$, and $p$, and is a decreasing function of $\tau_{M}$.
\end{prop}
Proposition \ref{thm:estimator.global_risk} follows straightforwardly from Proposition \ref{thm:estimator.global_twopoints} combined with the fact that with high probability, the balls centered at the bottleneck points $q_1$ and $q_2$ with radii $\mathcal{O}(n^{-1/d})$ both contain a sample point of $\X_n$.

\subsection{Local Case}\label{subsec:estimator.local}

Consider now the local case, that is, there exists $q_0 \in M$ and $v_0 \in T_{q_0} M$ such that the geodesic $\gamma_0 = \gamma_{q_0,v_0}$ has second derivative $\norm{\gamma_0''(0)} = {1}/{\tau_M}$ (Theorem~\ref{thm:geometry.local_global}). 
Estimating $\tau_M$ boils down to estimating the curvature of $M$ at $q_0$ in the direction $v_0$.

We first relate directional curvature to the increment $\frac{\norm{y-x}^2}{2 d(y-x,T_x M)}$ involved in the estimator $\hat{\tau}$ \eqref{eq:estimator.estimator}.
Indeed, since the latter quantity is the radius of a sphere tangent at $x$ and passing through $y$ (Figure \ref{fig:tangent_ball_2d}), it approximates the radius of curvature in the direction $y-x$ when $x$ and $y$ are close.
For $x,y \in M$, we let $\gamma_{x \rightarrow y}$ denote the arc-length parametrized geodesic joining $x$ and $y$, with the convention $\gamma_{x \rightarrow y}(0) = x$.
\begin{lem}\label{lem:estimator.local_twopoints}
	Let $M\in\mathcal{M}_{\tau_{\min},L}^{d,D}$ with reach $\tau_{M}$
	and $\mathbb{X}\subset M$ be a subset. Let $x,y\in\mathbb{X}$ with $d_{M}(x,y)<\pi\tau_{M}$.
	Then,
	\[
	0
	\leq
	\frac{1}{\tau_{M}}-\frac{1}{\hat{\tau}(\mathbb{X})}
	\leq
	\frac{1}{\tau_{M}}-\frac{1}{\hat{\tau}(\{x,y\})}
	\leq
	\frac{1}{\tau_{M}}- \norm{ \gamma_{x \rightarrow y}''(0) } + \frac{1}{3}Ld_{M}(x,y).
	\]
\end{lem}

Let us now state how directional curvatures are stable with respect to
perturbations of the base point and the direction.
We let $\kappa_p$ denote the maximal directional curvature of $M$ at $p \in M$, that is,
\begin{align*}
\kappa_{p} = \sup_{v\in\mathcal{B}_{T_{p}M}(0,1)}\left\Vert \gamma_{p,v}''(0)\right\Vert
.
\end{align*}
\begin{lem}
	\label{lem:estimator.local_butterfly} 
Let $M\in\mathcal{M}_{\tau_{\min},L}^{d,D}$ with reach $\tau_{M}$ and $q_{0},x,y\in M$ be such that $x,y\in\mathcal{B}_{M}\left(q_{0},\frac{\pi\tau_{M}}{2}\right)$.
Let $\gamma_{0}$ be a geodesic such that $\gamma_{0}(0)=q_{0}$ and
$\left\Vert \gamma_{0}''(0)\right\Vert = \kappa_{q_0}$. Write
\begin{align*}
\theta_{x}:=\angle(\gamma_{0}'(0),\gamma_{q_{0}\to x}'(0)) 
, \hspace{1em}
\theta_{y}:=\angle(\gamma_{0}'(0),\gamma_{q_{0}\to y}'(0))
,
\end{align*}
and suppose that $|\theta_{x}-\theta_{y}|\geq\frac{\pi}{2}$.
Then,
\[
\left\Vert \gamma_{x\to y}''(0)\right\Vert \geq \kappa_{q_{0}}-(\kappa_{x}-\kappa_{q_{0}})-2L d_{M}(q_{0},x)-(2\kappa_{x}+6\kappa_{q_{0}})\sin^{2}(|\theta_{x}-\theta_{y}|).
\]
\end{lem}
In particular, geodesics in a neighborhood of $q_0$ with directions close to
$v_0$ have curvature close to $\frac{1}{\tau_M}$. 
Combining Lemma \ref{lem:estimator.local_twopoints} and Lemma
\ref{lem:estimator.local_butterfly} yields the following deterministic bound  
in the local case.

\begin{prop}
	\label{thm:estimator.local_twopoints} 
Let $M\in\mathcal{M}_{\tau_{\min},L}^{d,D}$
	be such that there exist $q_{0}\in M$ and a geodesic $\gamma_{0}$
	such that $\gamma_{0}(0)=q_{0}$ and $\left\Vert \gamma_{0}''(0)\right\Vert =\frac{1}{\tau_{M}}$.
	Let $\mathbb{X}\subset M$ and $x,y\in\mathbb{X}$ be such that $x,y\in\mathcal{B}_{M}\left(q_{0},\frac{\pi\tau_{M}}{2}\right)$.
Write
\begin{align*}
\theta_{x}:=\angle(\gamma_{0}'(0),\gamma_{q_{0}\to x}'(0)) 
, \hspace{1em}
\theta_{y}:=\angle(\gamma_{0}'(0),\gamma_{q_{0}\to y}'(0))
,
\end{align*}
and suppose that $|\theta_{x}-\theta_{y}|\geq\frac{\pi}{2}$.
Then, 
\begin{align*}
0  \leq
\frac{1}{\tau_{M}}-\frac{1}{\hat{\tau}(\mathbb{X})}
&\leq
\frac{1}{\tau_{M}}-\frac{1}{\hat{\tau}(\{x,y\})}\\
& \leq\frac{8\sin^{2}(|\theta_{x}-\theta_{y}|)}{\tau_{M}} +L\left(\frac{1}{3}d_{M}(x,y)+2 d_{M}(q_{0},x)\right).
\end{align*}

\end{prop}

In other words, since the reach boils down to directional curvature in the local case, $\hat{\tau}$ performs well if it is given as input a pair of points $x,y$ which are close to the point $q_0$ realizing the reach, and almost aligned with the direction of interest $v_0$.
Note that the error bound in the local case (Proposition \ref{thm:estimator.local_twopoints}) is very similar to that of the global case (Proposition \ref{thm:estimator.global_twopoints}) with an extra alignment term $\sin^{2}(|\theta_{x}-\theta_{y}|)$
.
This alignment term appears since, in the local case, the reach arises from directional curvature $\tau_M = \norm{\gamma_{q_0,v_0}''(0)}$ (Theorem \ref{thm:geometry.local_global}). Hence, it is natural that the accuracy of $\hat{\tau}(\X)$ depends on how precisely $\X$ samples the neighborhood of $q_0$ in the particular direction $v_0$.

Similarly to the analysis of the global case, the deterministic bound in Proposition
\ref{thm:estimator.local_twopoints} yields a bound on the risk of $\hat{\tau}(\X_n)$ when $\X_n = \left\{X_1,\ldots,X_n\right\}$ is random.

\begin{prop}\label{thm:estimator.local_risk}
	Let $P \in \mathcal{P}_{\tau_{\min},L,f_{\min}}^{d,D}$ and $M=supp(P)$.
Suppose there exists $q_{0}\in M$ and a geodesic $\gamma_{0}$ with $\gamma_{0}(0)=q_{0}$ and $\norm{\gamma_{0}''(0)}=\frac{1}{\tau_{M}}$.
	Then,
	\[
	\mathbb{E}_{P^{n}}
	\left[
	\left|
	\frac{1}{\tau_{M}}-\frac{1}{\hat{\tau}(\X_n)}
	\right|^{p}
	\right]
	\leq 
	C_{\tau_{\min},d,L,f_{\min},p}n^{-\frac{2p}{3d-1}},
	\]
	where $C_{\tau_{\min},d,L,f_{\min},p}$ depends only on $\tau_{\min}, d, L , f_{\min}$ and $p$.
	
\end{prop}
This statement follows from Proposition \ref{thm:estimator.local_twopoints} together with the estimate of the probability of two points being drawn in a neighborhood of $q_0$ and subject to an alignment constraint.

Proposition \ref{thm:estimator.global_risk} and \ref{thm:estimator.local_risk} yield a convergence rate of $\hat{\tau}(\X_n)$ which is slower in the local case than in the global case. 
Recall that from Theorem \ref{thm:geometry.local_global}, the reach pertains to the size of a bottleneck structure in the global case, and to maximum directional curvature in the local case. 
To estimate the size of a bottleneck, observing two points close to each point in the bottleneck gives a good approximation.
However, for approximating maximal directional curvature, observing two points close to the curvature attaining point is not enough, but they should also be aligned with the highly curved direction. 
Hence, estimating the reach may be more difficult in the local case, and the difference in the convergence rates of Proposition \ref{thm:estimator.global_risk} and \ref{thm:estimator.local_risk} accords with this intuition.

Finally, let us point out that in both cases, neither the convergence rates nor the constants depend on the ambient dimension $D$.

\section{Minimax Estimates}\label{sec:minimax}
In this section we derive bounds on the minimax risk $R_n$ of
the estimation of the reach over the class $\mathcal{P}^{d,D}_{\tau_{\min},L,f_{\min}}$, that is
\begin{align}\label{eqn:minimax_risk_definition}
R_n = \inf_{\hat{\tau}_n} 
\sup_{P \in \mathcal{P}^{d,D}_{\tau_{\min},L,f_{min}}}
\E_{P^n} 
\left| 
\frac{1}{\tau_P} - \frac{1}{\hat{\tau}_n} 
\right|^p,
\end{align}
where the infimum ranges over all estimators $\hat{\tau}_n\bigl((X_1,T_{X_1}),\ldots,(X_n,T_{X_n})\bigr)$ based on an i.i.d. sample of size $n$ with the knowledge of the tangent spaces at sample points.
The minimax risk $R_n$ corresponds to the best expected risk that an estimator, based on $n$ samples, can achieve uniformly over the model $\mathcal{P}^{d,D}_{\tau_{\min},L,f_{min}}$ without the knowledge of the underlying distribution $P$.

The rate of convergence of the plugin estimator $\hat{\tau}_n = \hat{\tau}(\X_n)$ studied in the previous section leads to an upper
bound on $R_n$, which we state here for completeness.  
\begin{thm}\label{thm:upper_bound}
For all $n\geq 1$,
\[
R_n
\leq 
C_{\tau_{\min},d,L,f_{\min},p}n^{-\frac{2p}{3d-1}},
\]
for some constant $C_{\tau_{\min},d,L,f_{\min},p}$ depending only on $\tau_{\min}, d, L , f_{\min}$ and $p$. 
\end{thm}
We now focus on deriving a lower bound on the minimax risk $R_n$.
The method relies on an application of Le Cam's Lemma \cite{Yu1997}. In what follows, let
\begin{align*}
TV\left(P,P'\right) 
= 
\frac{1}{2} \int |d P - dP'|
\end{align*}
denote the total variation distance between $P$ and $P'$, where $dP,dP'$ denote the respective densities of $P,P'$ with respect to any dominating measure.
Since 
$|x-z|^p + |z-y|^p \geq 2^{1-p} |x-y|^p$
, the following version of Le Cam's lemma results from \cite[Lemma 1]{Yu1997} and $(1 - TV(P^n,P'^n)) \geq (1-TV(P,P'))^n$.
\begin{lem}[Le Cam's Lemma]\label{lecam_lemma}
Let $P,P' \in  \mathcal{P}^{d,D}_{\tau_{\min},L,f_{\min}}$ with respective supports $M$ and $M'$. 
Then for all $n \geq 1$,
\[
R_n 
\geq
\frac{1}{2^{p}}\left|\frac{1}{\tau_M}-\frac{1}{\tau_{M'}}\right|^p \left(1-TV(P,P')\right)^{n}.
\]
\end{lem}
Lemma \ref{lecam_lemma} states that in order to derive a lower bound on $R_n$
one needs to consider distributions (hypotheses) in the model that are stochastically close to each other --- i.e. with small total variation distance
--- but for which the associated reaches are as different as possible.
A lower bound on the minimax risk over
$\mathcal{P}^{d,D}_{\tau_{min},L,f_{min}}$ requires the hypotheses to belong to
the class. 
Luckily, in our problem it will be enough to construct hypotheses from the
simpler class  $\mathcal{Q}^{d,D}_{\tau_{\min},L,f_{\min}}$.
Indeed, we have the following isometry result between
$\mathcal{Q}^{d,D}_{\tau_{\min},L,f_{\min}}$ and $\mathcal{P}^{d,D}_{\tau_{min},L,f_{min}}$ for the total variation distance, as proved in Section \ref{sec:appendix:TV_lemmas}.
We use here the notation of Definition \ref{def:model_with_known_tangent_spaces}

\begin{lem}\label{thm:total_variation_identiy}
Let $Q,Q' \in \mathcal{Q}^{d,D}_{\tau_{\min},L,f_{\min}}$ be distributions on 
$\R^D$ with associated distributions $P,P' \in \mathcal{P}^{d,D}_{\tau_{\min},L,f_{\min},}$ on $\R^D \times \mathbb{G}^{d,D}$. Then,
\[
TV\left(P,P'\right) =  TV\left(Q,Q'\right).
\]
\end{lem}

In order to construct hypotheses in
$\mathcal{Q}^{d,D}_{\tau_{\min},L,f_{\min}}$ we take advantage of the fact that
the class $\mathcal{M}^{d,D}_{\tau_{min},L}$ has good stability properties,
which we now describe.
Here, since submanifolds do not have natural parametrizations, the notion of perturbation can be well formalized using diffeomorphisms of the ambient space $\R^D \supset M$. 
Given a smooth map $\Phi: \R^D \rightarrow \R^D$,  we denote by $d_x^i \Phi$ its
differential of order $i$ at $x$. Given a tensor field $A$ between Euclidean
spaces, let $\norm{A}_{op} = \sup_x \norm{A_x}_{op}$, where $\norm{A_x}_{op}$ is the operator norm induced by the Euclidean norm.
The  next result states, informally, that  the reach
and geodesics third derivatives of a submanifold that is perturbed by a
diffeomorphism that is $\mathcal{C}^3$-close to the identity map do not change
much. The proof of Proposition \ref{thm:diffeomorphism_stability} can be found in Section \ref{sec:appendix:construction}.
\begin{prop}
\label{thm:diffeomorphism_stability}
Let $M \in  \mathcal{M}^{d,D}_{\tau_{min}L}$ be fixed,
and let $\Phi : \R^D \rightarrow \R^D$ be a global $\mathcal{C}^3$-diffeomorphism.
If $\norm{I_D - d \Phi}_{op}$, $\norm{d^2 \Phi}_{op}$ and $\norm{d^3 \Phi}_{op}$ are small enough, then $M' = \Phi(M) \in \mathcal{M}^{d,D}_{\frac{\tau_{min}}{2},2L}$.
\end{prop}
\begin{figure}[h]
\centering
\includegraphics[trim = 0mm 90mm 0mm 0mm, clip= True, width=0.7\textwidth]{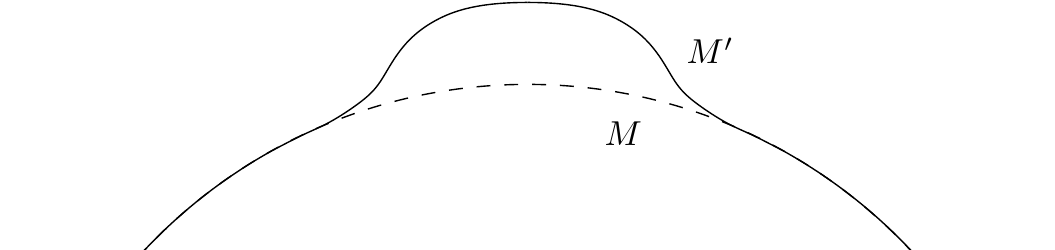}
\caption{Hypotheses of Proposition \ref{thm:two_hypotheses_construction}.}
\label{fig:two_hypotheses_simplified}
\end{figure}
Now we construct the two hypotheses $Q,Q'$ as follows (see Figure \ref{fig:two_hypotheses_simplified}).
Take $M$ to be a $d$-dimensional sphere and $Q$ to be the uniform distribution on it. 
Let $M' = \Phi(M)$, where  $\Phi$ is a bump-like diffeomorphism having the curvature of $M'$ to be different of that of $M$ in some small neighborhood. Finally, let $Q'$ be the uniform
distribution on $M'$. The proof of Proposition \ref{thm:two_hypotheses_construction} is to be found in Section \ref{sec:appendix:construction}.
\begin{prop}\label{thm:two_hypotheses_construction}
Assume that $L \geq  ({2\tau_{min}^2})^{-1}$ and $f_{min} \leq ({2^{d+1}\tau_{min}^d \sigma_d})^{-1}$. 
Then for $\ell > 0$ small enough, there exist $Q,Q' \in \mathcal{Q}^{d,D}_{\tau_{\min},L,f_{min}}$ with respective supports $M$ and $M'$ such that
\[
\left| \frac{1}{\tau_M} - \frac{1}{\tau_{M'}} \right| \geq 
c_d \frac{\ell}{\tau_{min}^2}
\text{~~ and ~~}
TV\left( Q,Q' \right) 
\leq 
12 \left(\frac{\ell}{2 \tau_{min}}\right)^d.
\]
\end{prop}

Hence, applying Lemma \ref{lecam_lemma} with the hypotheses $P,P'$ associated to $Q,Q'$ of Proposition \ref{thm:two_hypotheses_construction}, and taking
$ 
12 \left({\ell}/{2 \tau_{min}}\right)^d
=
1/n$, together with Lemma \ref{thm:total_variation_identiy}, yields the following lower bound.
\begin{prop}\label{thm:lower_bound_thm}
Assume that $L \geq  ({2\tau_{min}^2})^{-1}$ and $f_{min} \leq ({2^{d+1}\tau_{min}^d \sigma_d})^{-1}$.
Then for $n$ large enough,
\[
R_n
\geq 
\left(\frac{c_d}{\tau_{min}}\right)^p
n^{-p/d},
\]
where $c_{d}>0$ depends only on $d$.
\end{prop}
Here, the assumptions on the parameters $L$ and $f_{min}$ are necessary for the model to be rich enough.
Roughly speaking, they ensure at least that a sphere of radius $2\tau_{min}$ belongs to the model.

From Proposition \ref{thm:lower_bound_thm}, the plugin estimation $\hat{\tau}(\X_n)$ provably achieves the optimal rate in the global case (Theorem \ref{thm:estimator.global_risk}) up to numerical constants. 
In the local case (Theorem~\ref{thm:estimator.local_risk}) the rate obtained presents a gap, yielding a gap in the overall rate.
As explained above (Section \ref{subsec:estimator.local}), the slower rate in the local case is a consequence of the alignment required in order to estimate directional curvature.
Though, let us note that in the one-dimensional case $d=1$, the rate of Proposition \ref{thm:lower_bound_thm} matches the convergence rate of $\hat{\tau}(\X_n)$ (Theorem \ref{thm:upper_bound}). Indeed, for curves, the alignment requirement is always fulfilled.
Hence, the rate is exactly $n^{-p}$ for $d=1$, and $\hat{\tau}(\X_n)$ is minimax optimal.

Here, again, neither the convergence rate nor the constant depend on the ambient dimension $D$.

\section{Towards Unknown Tangent Spaces}
\label{sec:unknownTangent}

So far, in our analysis we have used the key assumption that both the point cloud and the tangent spaces were jointly observed. 
We now focus on the more realistic framework where only points are observed. 
We once again rely on the formulation of the reach given in Theorem
\ref{eqn:reach_as_a_supremum_federer} and consider a new plug-in estimator in
which the true tangent spaces are replaced by estimated ones.
Namely, given a point cloud $\X \subset \R^D$ and a family ${T} = \{{T}_x\}_{x\in \X}$ of linear subspaces of $\R^D$ indexed by $\X$, the estimator is defined as
\begin{align}\label{eqn:plugin_tangent_reach_definition}
{\hat{\tau}(\X,{T})} = \inf_{x \neq y \in \X} ~ \frac{\norm{y-x}^2}{2d(y-x,{T}_x)}.
\end{align}
In particular, $\hat{\tau}(\X) = \hat{\tau}(\X,T_{\X} M)$, where $T_\X M = \{{T}_x M\}_{x\in \X}$.
Adding uncertainty on tangent spaces in \eqref{eqn:plugin_tangent_reach_definition} does not change drastically the estimator as the formula is stable with respect to $T$. We state this result quantitatively in the following Proposition \ref{thm:tangent_stability}, the proof of which can be found in Section \ref{sec:appendix:stability_tangent}.
In what follows, the distance between two linear subspaces $U,V \in \mathbb{G}^{d,D}$ is measured with their principal angle $\norm{\pi_U - \pi_{V} }_{op}$.
\begin{prop}\label{thm:tangent_stability}
Let $\X \subset \R^D$
and $ T = \{T_x\}_{x\in \X}$, $\tilde{T} = \{\tilde{T}_x\}_{x\in \X}$ be two families of linear subspaces of $\R^D$ indexed by $\X$.
Assume $\X$ to be $\delta$-sparse, $T$ and $\tilde{T}$ to be $\theta$-close, 
in the sense that
\[
\underset{x\neq y \in \X}{\inf} \norm{y-x} \geq \delta
\hspace{1em} \text{ and } \hspace{1em}
\sup_{x \in \X} \Vert T_x - \tilde{T}_x \Vert_{op} \leq \sin \theta.
\]
Then,
\[
\left| 
\frac{1}{\hat{\tau}(\X,T)}
-
\frac{1}{\hat{\tau}(\X,\tilde{T})}
\right|
\leq \frac{2 \sin \theta}{\delta}.
\]
\end{prop}
In other words, the map $T \mapsto \hat{\tau}(\X,T)^{-1}$ is smooth, provided that the basis point cloud $\X$ contains no zone of accumulation at a too small scale $\delta>0$.
As a consequence, under the assumptions of Proposition \ref{thm:tangent_stability}, the bounds on $\bigl| {\hat{\tau}(\X)}^{-1} - {\tau_M}^{-1} \bigr|$ of Proposition \ref{thm:estimator.global_twopoints} and Proposition \ref{thm:estimator.local_twopoints} still hold  with an extra error term $2\sin \theta / \delta$ if we replace $\hat{\tau}(\X)$ by $\hat{\tau}(\X,T)$.

For an i.i.d. point cloud $\X_n$,
asymptotic and nonasymptotic rates of tangent space estimation derived in $\mathcal{C}^3$-like models can be found in \cite{Aamari17,Cheng16,Singer12}, yielding bounds on $\sin \theta$ of order $\left( \log n / n\right)^{1/d}$.
In that case, the typical scale of minimum interpoint distance is $\delta \asymp n^{-2/d}$, as stated in the asymptotic result Theorem 2.1 in \cite{Kanagawa92} for the flat case of $\R^d$.
However, the typical covering scale of $M$ used in the global case (Theorem \ref{thm:estimator.global_risk}) is $\varepsilon \asymp (1/n)^{1/d}$.
It appears that we can sparsify the point cloud $\X_n$ --- that is, removing accumulation points  --- while preserving the covering property at scale $\varepsilon = 2 \delta \asymp \left( \log n / n\right)^{1/d}$. 
This can be performed using the \textit{farthest point sampling algorithm} \cite[Section 3.3]{Aamari15}. 
Such a sparsification pre-processing allows to lessen the possible instability of $\hat{\tau}(\X_n,\cdot)^{-1}$. 
Though, 
whether the alignment property used in the local case (Theorem \ref{thm:estimator.local_risk}) is preserved under sparsification remains to be investigated.

\section{Conclusion and Open Questions}
\label{sec:conclusion}

In the present work, we gave new insights on the geometry of the reach. 
Inference results were derived in both deterministic and random frameworks. 
For i.i.d. samples, non-asymptotic minimax upper and lower bounds were derived under assumptions on the third order derivative of geodesic trajectories.
Let us conclude with some open questions.

\begin{itemize}[leftmargin=*]
\item 
Interestingly, the derivation of the minimax lower bound (Theorem \ref{thm:lower_bound_thm}) involves hypotheses that correspond to the local case, but yields the rate $n^{-p/d}$. But, on the upper bound side, this rate matches with that of the global case (Theorem~\ref{thm:estimator.global_risk}), the local case being slower (Theorem~\ref{thm:estimator.local_risk}).
The minimax upper and lower bounds given in Theorem \ref{thm:upper_bound} and Theorem~\ref{thm:lower_bound_thm} do not match. They are yet to be sharpened.
This results into minimax upper and lower bounds that do not match. They are yet to be sharpened.
\item
As mentioned earlier, Section \ref{sec:unknownTangent} is only a first step towards a framework where tangent spaces are unknown. A minimax upper bound in this case is still an open question.
Considering smoother $\mathcal{C}^k$ models ($k \geq 3$) such as those of \cite{Aamari17}, or data with additive noise would also be of interest.
\item In practice, since large reach ensures regularity, one may be interested with having a lower bound on the reach $\tau_M$. 
Studying the limiting distribution of the statistic $\hat{\tau}(\X_n)$ would allow to derive asymptotic confidence intervals for $\tau_M$.
\item Other regularity parameters such as local feature size \cite{Boissonnat14} and $\lambda$-reach \cite{ChazalL2005} could be relevant to estimate, as they are used as tuning parameters in computational geometry techniques.

\end{itemize}

\section*{Acknowledgments}
This collaboration was made possible by the associated team CATS (Computations And Topological Statistics) between DataShape and Carnegie Mellon University.
We thank warmly the anonymous reviewers for their insight, which led to various significant improvements of the paper.

\appendix

\section{Some Technical Results on the Model}
\subsection{Geometric Properties}

The following Proposition \ref{thm:geometry:injectivity_radius} garners geometric properties of submanifolds of the Euclidean space that are related to the reach. We will use them numerous times in the proofs.
\begin{prop}
	\label{thm:geometry:injectivity_radius}
Let $M \subset \R^D$ be a closed submanifold with reach $\tau_M > 0$.
\begin{enumerate}
	\item[(i)]  For all $p\in M$, we let $II_{p}$ denote the second fundamental
	form of $M$ at $x$. Then for all unit vector $v\in T_{p}M$, $\norm{II_{p}(v,v)}\leq\frac{1}{\tau_{M}}$.
	\item[(ii)]  The injectivity radius of $M$ is at least $\pi\tau_{M}$.
	\item[(iii)]  The sectional curvatures $K$ of $M$ satisfy $-\frac{2}{\tau_{M}^{2}}\leq  K \leq\frac{1}{\tau_{M}^{2}}$.
	\item[(iv)]  For all $p\in M$, the map $\exp_{p}:\accentset{\circ}{\mathcal{B}}_{T_{p}M}\left(0,\pi\tau_{M}\right)\rightarrow \accentset{\circ}{\mathcal{B}}_M\left(0,\pi\tau_{M}\right)$ is a diffeomorphism. Moreover, for all $\norm{v}<\frac{\pi\tau_{M}}{2\sqrt{2}}$ and $w\in T_{p}M$, 
\[
{\color{black}{
\left(1-\frac{\norm{v}^{2}}{6\tau_{M}^{2}}\right)\norm{w}
}}
\leq
\norm{d_{v}\exp_{p}\cdot w}
\leq
\left(1+\frac{\norm{v}^{2}}{\tau_{M}^{2}}\right)\norm{w}.
\]
	\item[(v)] For all $p\in M$ and $r\leq\frac{\pi\tau_{M}}{2\sqrt{2}}$, given any Borel set $A\subset\mathcal{B}_{T_{p}M}(0,r)\subset T_{p}M$,
	\[
	\left(1-\frac{r^{2}}{6\tau_{M}^{2}}\right)^{d}\mathcal{H}^{d}(A)\leq\mathcal{H}^{d}(\exp_{p}(A))\leq\left(1+\frac{r^{2}}{\tau_{M}^{2}}\right)^{d}\mathcal{H}^{d}(A).
	\]
\end{enumerate}
\end{prop}
\begin{proof}[Proof of Proposition \ref{thm:geometry:injectivity_radius}]
(i) is stated as in \cite[Proposition 2.1]{NiyogiSW2008},
yielding (ii) from \cite[Corollary 1.4]{Alexander06}.
(iii) follows using (i) again and the Gauss equation~\cite[p. 130]{DoCarmo92}. (iv) is derived from (iii) by a direct application of \cite[Lemma 8]{DyerVW2015}. (v) follows from (iv) and \cite[Lemma 6]{Arias13}. 	
	
\end{proof} 

\subsection{Comparing Reach and Diameter}
Let us prove Proposition \ref{thm:geometry:max_reach_volumic}. For this aim, we first state the following analogous bound on the (Euclidean) diameter $diam(M) = \sup_{x,y \in M} \norm{x-y}$.

\begin{lem}[Lemma 2 in \cite{Aamari15}]\label{thm:geometry:diameter_bound}
Let $M\subset \R^D$ be a connected closed $d$-dimensional manifold, and let $Q$ be a probability distribution having support $M$ with a density $f \geq f_{min}$ with respect to the Hausdorff measure on $M$. Then,
\[
diam(M) \leq \frac{C_d}{\tau_{M}^{d-1}f_{min}},
\]
for some constant $C_d>0$ depending only on $d$.
\end{lem}

\begin{prop}\label{thm:geometry:reach_vs_diameter}
If $K \subset \R^D$ is not homotopy equivalent to a point,
\[
\tau_K \leq \sqrt{\frac{D}{2(D+1)}}diam(K).
\]
\end{prop}
\begin{proof}[Proof of Proposition \ref{thm:geometry:reach_vs_diameter}]
Combine Lemma~\ref{thm:geometry:optimal_hull_inclusion} and Lemma~\ref{thm:geometry:reach_vs_hull}. 
\end{proof} 
Let us recall that for two compact subsets $A,B \subset \R^D$, the Hausdorff distance~\cite[p. 252]{Burago01} between them is defined by
\[
d_H(A,B) 
=
\max 
\bigl\{ 
\sup_{a\in A} d(a,B)
,
\sup_{b \in B} d(b,A)
\bigr\}.
\]
We denote by $conv(\cdot)$ the closed convex hull of a set.
\begin{lem}\label{thm:geometry:optimal_hull_inclusion}
For all $K \subset \R^D$, $d_H\left(K,conv(K)\right) \leq \sqrt{\frac{D}{2(D+1)}}diam(K)$.
\end{lem}

\begin{proof}[Proof of Lemma \ref{thm:geometry:optimal_hull_inclusion}]
It is a straightforward corollary of Jung's Theorem 2.10.41 in~\cite{Federer69}, which states that $K$ is contained in a (unique) closed ball with (minimal) radius at most $\sqrt{\frac{D}{2(D+1)}}diam(K)$.
\end{proof}

\begin{lem}\label{thm:geometry:reach_vs_hull}
If $K \subset \R^D$ is not homotopy equivalent to a point, then
$\tau_K \leq d_H\left(K,conv(K) \right)$.
\end{lem}

\begin{proof}[Proof of Lemma \ref{thm:geometry:reach_vs_hull}] 
Let us prove the contrapositive. For this, assume that $\tau_K > d_H\left(K,conv(K)\right)$. Then,
\begin{align*}
conv(K)
&\subset 
\bigcup_{x\in K} \mathcal{B}\left(x, d_H\left(K,conv(K)\right)\right)
\subset
\bigcup_{x\in K} \accentset{\circ}{\mathcal{B}}\left(x, \tau_K\right)
\subset Med(K)^c.
\end{align*}
Therefore, the map $\pi_K: conv(K) \rightarrow K $ is well defined and continuous, so that $K$ is a retract of $conv(K)$ (see \cite[Chapter 0]{Hatcher02}). Therefore, $K$ is homotopy equivalent to a point, since the convex set $conv(K)$ is.
\end{proof}

We are now in position to prove Proposition \ref{thm:geometry:max_reach_volumic}.

\begin{proof}[Proof of Proposition \ref{thm:geometry:max_reach_volumic}]
From \cite[Theorem 3.26]{Hatcher02}, $M$ has a non trivial homology group of dimension $d$ over $\mathbb{Z}/2 \mathbb{Z}$, so that it cannot be homotopy equivalent to a point. Therefore, Proposition \ref{thm:geometry:reach_vs_diameter} yields $\tau_M \leq diam(M)$, and we conclude by applying the bound $diam(M) \leq C_d/(\tau_M^{d-1} f_{min})$ given by Lemma \ref{thm:geometry:diameter_bound}.
\end{proof}

\section{Geometry of the Reach}
\label{sec:appendix:geometry_of_the_reach}

\begin{lem}
\label{lem:geometry_circledisk_intersection}
Let $V \subset \R^D$ be a $2$-dimensional affine space and $q_{1},q_{2},z,p \in V$ be such that
$
\norm{ p-q_{1}}
=
\norm{
p-q_{2}
}
= r_p
\text{ and }
\norm{
z-q_{1}
}
=
\norm{
z-q_{2}
}
= r_z.
$
If $r_p < r_z$, then 
	\[
	V\cap\partial\mathcal{B}_{\mathbb{R}^{D}}(z,r_{z})\cap\mathcal{B}_{\mathbb{R}^{D}}(p,r_{p})=c_{z}(q_{1},q_{2}),
	\]
	where $c_{z}(q_{1},q_{2})$ is the shorter arc of the circle with center $z$
	and endpoints as $q_{1}$ and $q_{2}$.
	
\end{lem}

\begin{figure}[h!]
	\centering
	\includegraphics{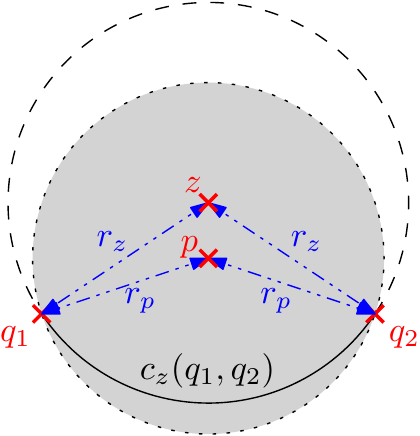}
	\caption{Layout of Lemma \ref{lem:geometry_circledisk_intersection}.
	}
	\label{fig:geometry_circledisk_intersection}
\end{figure}

\begin{proof}[Proof of Lemma \ref{lem:geometry_circledisk_intersection}]
	
	Since everything is intersected with the $2$-dimensional space $V$,
	we can assume that $D=2$ without loss of generality. For short, we write $K = \partial\mathcal{B}_{\mathbb{R}^{2}}(z,r_{z})\cap\mathcal{B}_{\mathbb{R}^{2}}(p,r_{p})$.

	First note that $\{q_1,q_2\} \subset K$, so that $K \neq\emptyset$.
	Furthermore, for all $x\in K$,
	$d(x,Med(\partial\mathcal{B}_{\mathbb{R}^{2}}(z,r_{z}))=r_{z}>r_{p}$,
	so that $\tau_K > r_p$ from \cite[Lemma 3.4 (i)]{RatajZ2017}.
	Hence, applying \cite[Lemma 3.4 (ii)]{RatajZ2017}, we get that $K$ is contractible. In particular, $K$ is connected.
	
	Since $K$ is a closed connected subset of the circle $\partial\mathcal{B}_{\mathbb{R}^{2}}(z,r_{z})$, $K$ is an arc of a circle. Let $c_{1},c_{2}$ denote its endpoints.
	
	Let us now show that $\{c_{1},c_{2}\}\subset\partial\mathcal{B}_{\mathbb{R}^{2}}(z,r_{z})\cap\partial\mathcal{B}_{\mathbb{R}^{2}}(p,r_{p})$, or equivalently that $\left\Vert c_{1}-p\right\Vert =\left\Vert c_{2}-p\right\Vert =r_{p}$.
	Indeed, if $x\in K$ is such that $\left\Vert x-p\right\Vert <r_{p}$ then there exists $r_{x}>0$
	such that $\mathcal{B}_{\mathbb{R}^{2}}(x,r_{x})\subset\mathcal{B}_{\mathbb{R}^{2}}(p,r_{p})$.
	Then $\partial\mathcal{B}_{\mathbb{R}^{2}}(z,r_{z})\cap\mathcal{B}_{\mathbb{R}^{2}}(x,r_{x})\neq\emptyset$,
	so $\partial\mathcal{B}_{\mathbb{R}^{2}}(z,r_{z})\cap\mathcal{B}_{\mathbb{R}^{2}}(x,r_{x})$
	is also an arc of a circle, and since $x\in\partial\mathcal{B}_{\mathbb{R}^{2}}(z,r_{z})$,
	$x$ cannot be an end point of the arc $\partial\mathcal{B}_{\mathbb{R}^{2}}(z,r_{z})\cap\mathcal{B}_{\mathbb{R}^{2}}(x,r_{x})$.
	
	The two circles $\partial\mathcal{B}_{\mathbb{R}^{2}}(z,r_{z})$ and $\partial\mathcal{B}_{\mathbb{R}^{2}}(p,r_{p})$ are different
	($r_{z}>r_{p}$), so their intersection contains at most two points.
	Since $q_{1}\neq q_{2}\in K = \partial\mathcal{B}_{\mathbb{R}^{2}}(z,r_{z})\cap\mathcal{B}_{\mathbb{R}^{2}}(p,r_{p})$,
	in fact $\{q_{1},q_{2}\}=\partial\mathcal{B}_{\mathbb{R}^{2}}(z,r_{z})\cap\partial\mathcal{B}_{\mathbb{R}^{2}}(p,r_{p})$.
	Consequently, $\{c_{1},c_{2}\}=\{q_{1},q_{2}\}$. That is, $q_{1}$ and $q_{2}$ are the endpoints of the arc $K$.
	
	Note that there are two arcs of the circle $\partial\mathcal{B}_{\mathbb{R}^{2}}(z,r_{z})$	with endpoints $q_{1}$ and $q_{2}$. Since $K = \partial\mathcal{B}_{\mathbb{R}^{2}}(z,r_{z})\cap\mathcal{B}_{\mathbb{R}^{2}}(p,r_{p})\subset\mathcal{B}_{\mathbb{R}^{2}}(p,r_{p})$
	and $r_{p}<r_{z}$, $K$ cannot contain two points at distance equal to $2r_{z}$. Hence, $K$	is the shorter arc of the circle $\partial\mathcal{B}_{\mathbb{R}^{2}}(z,r_{z})$
	with endpoints $q_{1}$ and $q_{2}$, which is exactly $c_{z}(q_{1},q_{2})$.

\end{proof}

\begin{lem}
	
	\label{lem:geometry_plane_outpoint}
	
	Let $V \subset \R^D$ be a $2$-dimensional affine space and $q_{1},q_{2},z,x \in V$. 
	Denote by $ L $ be the line passing $q_{1}$	and $q_{2}$. 
	Assume that $x,z\notin L $, and that the segment joining $z$ and $x$ intersects $ L $. 
	Let $p\in\mathbb{R}^{D}$ be such that $\left\Vert p-q_{1}\right\Vert =\left\Vert z-q_{1}\right\Vert $
	and $\left\Vert p-q_{2}\right\Vert =\left\Vert z-q_{2}\right\Vert $.
	Then 
$ \left\Vert p-x\right\Vert \leq\left\Vert z-x\right\Vert ,
$
and the equality holds if and only if $p=z$. 
	
\end{lem}

\begin{proof}[Proof of Lemma \ref{lem:geometry_plane_outpoint}]
	
	Let $y$ denote the intersection point of $ L $ and the line segment between $z$ and $x$. 
	Since $\left\Vert p-q_{1}\right\Vert =\left\Vert z-q_{1}\right\Vert $
	and $\left\Vert p-q_{2}\right\Vert =\left\Vert z-q_{2}\right\Vert $,
	\begin{align*}
	\cos(\angle(p-q_{1},q_{2}-q_{1}))
	&=
	\cos(\angle(z-q_{1},q_{2}-q_{1}))
	\\
	&=\frac{\left\Vert z-q_{1}\right\Vert ^{2}+\left\Vert q_{2}-q_{1}\right\Vert ^{2}-\left\Vert z-q_{2}\right\Vert ^{2}}{2\left\Vert z-q_{1}\right\Vert \left\Vert q_{2}-q_{1}\right\Vert },
	\end{align*}
from which we derive
	\begin{align*}
	\left\Vert p-y\right\Vert^2  & =\left\Vert p-q_{1}\right\Vert ^{2}+\left\Vert y-q_{1}\right\Vert ^{2}-2\left\Vert p-q_{1}\right\Vert \left\Vert y-q_{1}\right\Vert \cos(\angle(p-q_{1},q_{2}-q_{1}))
	\\
	& =
	\left\Vert z-q_{1}\right\Vert ^{2}+\left\Vert y-q_{1}\right\Vert ^{2}-2\left\Vert z-q_{1}\right\Vert \left\Vert y-q_{1}\right\Vert \cos(\angle(z-q_{1},q_{2}-q_{1}))
	\\
	& =\left\Vert z-y\right\Vert^2 .
	\end{align*}
	Using the fact that $y$ belongs to the segment joining $x$ and $z$, we get
	\begin{align*}
	\left\Vert z-x\right\Vert  & =\left\Vert z-y\right\Vert +\left\Vert y-x\right\Vert \\
	& =\left\Vert p-y\right\Vert +\left\Vert y-x\right\Vert \\
	& \geq\left\Vert p-x\right\Vert
	.
	\end{align*}
	Finally, note that since $x,z\notin L $ and $y\in L $, the equality holds if and only if $\angle(x-y,p-y)=\pi$. But $\left\Vert p-y\right\Vert =\left\Vert z-y\right\Vert $
	and $x,y$, and $z$ are colinear, so this is possible if and only if
	$p=z$.
	
\end{proof}	

The following lemma can be seen as an extension of \cite[Lemma 3.4 (i)]{RatajZ2017}.

\begin{lem}
	
	\label{lem:geometry_reach_intersection_balls}
	
	Let $A\subset\mathbb{R}^{D}$ be a set with positive reach $\tau_{A}>0$,
	and let $\{B_{i}\}_{i\in I}$ be a collection of balls indexed by
	$I$. Suppose $\bigcap_{i\in I}B_{i}\cap A$ is nonempty. Let $r_{i}$
	be the radius of $B_{i}$, and suppose $r_{i}<\tau_{A}$ for all $i\in I$.
	\begin{enumerate}
		\item[(i)]  If $I$ is finite, then 
		$
		\tau_{\bigcap_{i\in I}B_{i}\cap A}>\min_{i\in I}r_{i}.
		$
		\item[(ii)]  If $I$ is countably infinite, then 
		$
		\tau_{\bigcap_{i\in I}B_{i}\cap A}\geq\inf_{i\in I}r_{i}.
		$
	\end{enumerate}
\end{lem}

\begin{proof}[Proof of Lemma \ref{lem:geometry_reach_intersection_balls}]
	
	\begin{enumerate}[leftmargin=*]
		\item[(i)] 
	Since $I$ is finite, we can assume that $I = \{1,\ldots,k\}$ and that the sequence $(r_{i})_{1\leq i \leq k}$ is nonincreasing. We use an induction
	on $k$:
	\begin{itemize}[leftmargin=*]
		\item If $k=1$, since for all $x\in A\cap B_{1}$, $d(x,Med(A))\geq\tau_{A}>r_{1}$,
		\cite[Lemma 3.4 (i)]{RatajZ2017} gives that $\tau_{B_{1}\cap A}>r_{1}$. 
		\item Suppose now that $\tau_{\bigcap_{i-1}^{j}B_{i}\cap A}>r_{j}$ for some
		$j<k$. Then for all $x\in\bigcap_{i=1}^{j+1}B_{i}\cap A=\left(\bigcap_{i=1}^{j}B_{i}\cap A\right)\cap B_{j+1}$,
		$d\left(x,Med\left(\bigcap_{i=1}^{j}B_{i}\cap A\right)\right)>r_{j}\geq r_{j+1}$.
		Applying again \cite[Lemma 3.4 (i)]{RatajZ2017} gives 
		\[
		\tau_{\bigcap_{i=1}^{j+1}B_{i}\cap A}=\tau_{\left(\bigcap_{i=1}^{j}B_{i}\cap A\right)\cap B_{j+1}}>r_{j+1}=\min_{1 \leq i \leq j+1}r_{i}.
		\]
	\end{itemize}
	By induction on $k$, we get the result.
	
	\item[(ii)] 
	
	Note that if $\inf_{i\in I}r_{i}=0$, there is nothing to prove. Hence we only consider the case where $\inf_{i\in I}r_{i}>0$.
	
	Since $I$ is countable, we can assume that $I=\mathbb{N}$. For $k\in\mathbb{N}$,
	let $C_{k}:=\bigcap_{i=1}^{k}B_{i}\cap A$. In particular,
	$
	C_\infty := \cap_{k=1}^{\infty}C_{k} = \cap_{i=1}^{\infty}B_{i}\cap A.
	$
	From the finite case (i), 
	\[
	\tau_{C_{k}}>\min_{1 \leq i\leq k}r_{i}\geq\inf_{i\in\mathbb{N}}r_{i}.
	\]
	Now, since $\{C_{k}\}_{k=1}^{\infty}$ is a decreasing sequence of
	sets, the distance functions $d(\cdot,C_{k})$ converges to $d(\cdot,C_\infty)$.
	As the distance functions $d(\cdot,C_{k})$ are continuous,
	there convergence is uniform on any compact subset of $\mathbb{R}^{D}$.
	Hence,~\cite[Theorem 5.9]{Federer1959} yields 
	$\tau_{C_\infty}\geq\inf_{i\in\mathbb{N}}r_{i},
	$
	which concludes the proof.
\end{enumerate}
\end{proof}

\begin{proof}[Proof of Lemma \ref{thm:geometry.circlearc}]
	Let $p_{0}:=\frac{z_{0}+q_{1}+q_{2}}{3}$ and $\tau_{0}=\left\Vert p_0 - q_{1}\right\Vert <\tau_{M}$.
	Consider the subset $C_0$ of the median hyperplane of $q_1$ and $q_2$ defined by
	\[
	C_{0}
	:=
	\left\{ p\in\mathbb{R}^{D} | \left\Vert p-q_{1}\right\Vert =\left\Vert p-q_{2}\right\Vert \in(\tau_{0},\tau_{M})\right\}
	,
	\]
	and let $\{p_{i}\}_{i\in\mathbb{N}}\subset C_{0}$ be its countable dense subset.
	Write
	$\tau_{i}:=\left\Vert p_{i}-q_{1}\right\Vert $ and $B_{i}:=\mathcal{B}_{\mathbb{R}^{D}}(p_{i},\tau_{i})$.
Let $A_{\infty}:=\bigcap_{k=0}^{\infty}B_{k}$. Note that $\{q_{1},q_{2}\}\subset M\cap A_{\infty}$ 
	which implies that $M\cap A_{\infty}$ is nonempty. 
	Note also that by definition, $\tau_{i}\in(\tau_{0},\tau_{M})$ for all $i\in\mathbb{N}\cup\{0\}$.
	Hence from Lemma \ref{lem:geometry_reach_intersection_balls}
	(ii), $\tau_{M\cap A_{\infty}}\geq\tau_{0}.$
In addition, $M\subset\mathbb{R}^{D}\setminus\accentset{\circ}{\mathcal{B}}_{\mathbb{R}^{D}}(z_{0},\tau_{M})$, so that
	\[
	\{q_{1},q_{2}\}\subset M\cap A_{\infty}\subset A_{\infty}\setminus\accentset{\circ}{\mathcal{B}}_{\mathbb{R}^{D}}(z_{0},\tau_{M}).
	\]
Note that it is sufficient to show that $A_{\infty}\setminus\accentset{\circ}{\mathcal{B}}_{\mathbb{R}^{D}}(z_{0},\tau_{M})=c_{z_{0}}(q_{1},q_{2})$ to conclude the proof. Indeed, since $\tau_{M\cap A_{\infty}}\geq\tau_{0}>\frac{\left\Vert q_{2}-q_{1}\right\Vert }{2}$
	and $\emptyset\neq M\cap A_{\infty}\subset c_{z_{0}}(q_{1},q_{2})\subset\mathcal{B}_{\mathbb{R}^{D}}\left(\frac{q_{1}+q_{2}}{2},\frac{\left\Vert q_{2}-q_{1}\right\Vert }{2}\right)$, \cite[Lemma 3.4 (ii)]{RatajZ2017} implies that $M\cap A_{\infty}$ is contractible. In other words, $M\cap A_{\infty}$ is a contractible subset of the shorter arc of a circle $c_{z_{0}}(q_{1},q_{2})$ containing its endpoints $q_{1}$ and $q_{2}$, and hence $M\cap A_{\infty}=c_{z_{0}}(q_{1},q_{2})$. Therefore, 
$c_{z_{0}}(q_{1},q_{2})\subset M,$ which concludes the proof.
	
	It is left to show that $A_{\infty}\setminus\accentset{\circ}{\mathcal{B}}_{\mathbb{R}^{D}}(z_{0},\tau_{M})=c_{z_{0}}(q_{1},q_{2})$.
	To this aim, let us write $V:=z_{0}+span\left\{ q_{1}-z_{0},q_{2}-z_{0}\right\} $ for the	$2$-dimensional plane passing through $q_{1}, q_{2}$, and $z_{0}$.
	Then $\tau_0 = \left\Vert p_{0}-q_{1}\right\Vert =\left\Vert p_{0}-q_{2}\right\Vert <\left\Vert z_{0}-q_{1}\right\Vert =\left\Vert z_{0}-q_{2}\right\Vert =\tau_{M}$,
	and hence from Lemma \ref{lem:geometry_circledisk_intersection},
	$c_{z_{0}}(q_{1},q_{2})$ can be represented as 
	\begin{equation}
	c_{z_{0}}(q_{1},q_{2})=V\cap\partial\mathcal{B}_{\mathbb{R}^{D}}(z_{0},\tau_{M})\cap\mathcal{B}_{\mathbb{R}^{D}}(p_{0},\tau_{0}).\label{eq:geometry.circlearc_representation}
	\end{equation}
	The proof will hence be complete as soon as we have showed the equality
	\[
	A_{\infty}\setminus\accentset{\circ}{\mathcal{B}}_{\mathbb{R}^{D}}(z_{0},\tau_{M})
	=
	V\cap\partial\mathcal{B}_{\mathbb{R}^{D}}(z_{0},\tau_{M})\cap\mathcal{B}_{\mathbb{R}^{D}}(p_{0},\tau_{0}),
	\]
	which we tackle by showing the two inclusions.
	\begin{itemize}[leftmargin=*]
	\item \textit{(Direct inclusion)}
Let $x\in \R^D \setminus \mathcal{B}_{\mathbb{R}^{D}}(z_{0},\tau_{M})$.
	Since $\left\Vert z_{0}-q_{1}\right\Vert =\left\Vert z_{0}-q_{2}\right\Vert =\tau_{M}$, their exists $p_{i}$ satisfying $\left\Vert p_{i}-z_{0}\right\Vert <\frac{\left\Vert z_{0}-x\right\Vert -\tau_{M}}{2}$.
	Then,
	\[
	\left\Vert p_{i}-x\right\Vert \geq\left\Vert z_{0}-x\right\Vert -\left\Vert p_{i}-z_{0}\right\Vert \geq\frac{\left\Vert z_{0}-x\right\Vert +\tau_{M}}{2}>\tau_{M}>\left\Vert p_{i}-q_{1}\right\Vert ,
	\]
	so that $x\notin B_{i}=\mathcal{B}_{\mathbb{R}^{D}}\left(p_{i},\left\Vert p_{i}-q_{1}\right\Vert \right)$,
	and $x\notin A_{\infty}=\bigcap_{i=1}^{\infty}B_{i}$ as well. Hence
	this implies that 
	\begin{equation}
	(\mathbb{R}^{D}\setminus\mathcal{B}_{\mathbb{R}^{D}}(z_{0},\tau_{M}))\cap A_{\infty}=\emptyset.\label{eq:geometry.circlearc_outball}
	\end{equation}	
	Let now $x\in(\partial\mathcal{B}_{\mathbb{R}^{D}}(z_{0},\tau_{M}))\setminus V$.
	Since $x,q_{1}$, and $q_{2}$ are not colinear, we can find $p'\in V_x =  x+span\left\{ q_{1}-x,q_{2}-x\right\}$
	such that $\left\Vert p'-q_{1}\right\Vert =\left\Vert p'-q_{2}\right\Vert =\tau_{M}$ and the line segment between $p'$ and $x$ intersects the line $L$ passing by $q_{1}$ and $q_{2}$. 
	Then $q_{1}, q_{2}, x,$ and $p'$
	are lying on a $2$-dimensional plane $V_x$, and $x \notin L$. 
	Also, $x\notin V$, $q_{1},q_{2}\in V$, and $\left\Vert p'-q_{1}\right\Vert =\left\Vert p'-q_{2}\right\Vert =\tau_{M}>\frac{\left\Vert q_{1}-q_{2}\right\Vert }{2}$
	implies that $p'\notin V$, and hence $p'\neq z_{0}$. Hence from
	Lemma~\ref{lem:geometry_plane_outpoint}, 
	\[
	\left\Vert p'-x\right\Vert >\left\Vert z_{0}-x\right\Vert =\tau_{M}.
	\]
	Now, since $\left\Vert p'-q_{1}\right\Vert =\left\Vert p'-q_{2}\right\Vert =\tau_{M}$,
	there exists $p_{i'}$ be satisfying $\left\Vert p_{i'}-p'\right\Vert <\frac{\left\Vert p'-x\right\Vert -\tau_{M}}{2}$.
	Then 
	\[
	\left\Vert p_{i'}-x\right\Vert \geq\left\Vert p'-x\right\Vert -\left\Vert p_{i'}-p'\right\Vert \geq\frac{\left\Vert p'-x\right\Vert +\tau_{M}}{2}>\tau_{M}>\left\Vert p_{i'}-q_{1}\right\Vert ,
	\]
	and hence $x\notin B_{i'}=\mathcal{B}_{\mathbb{R}^{D}}\left(p_{i'},\left\Vert p_{i'}-q_{1}\right\Vert \right)$, $x\notin A_{\infty}=\bigcap_{i=1}^{\infty}B_{i}$ as well. Hence
	this implies that 
	\begin{equation}
	((\partial\mathcal{B}_{\mathbb{R}^{D}}(z_{0},\tau_{M}))\setminus V)\cap A_{\infty}=\emptyset.\label{eq:geometry.circlearc_outplane}
	\end{equation}
	
	Finally, by construction, $A_{\infty}\subset\mathcal{B}_{\mathbb{R}^{D}}(p_{0},\tau_{0})$. Combining this last inclusion with \eqref{eq:geometry.circlearc_outball} and \eqref{eq:geometry.circlearc_outplane} yields the desired inclusion
	\begin{equation}
	A_{\infty}\setminus\accentset{\circ}{\mathcal{B}}_{\mathbb{R}^{D}}(z_{0},\tau_{M})\subset V\cap\partial\mathcal{B}_{\mathbb{R}^{D}}(z_{0},\tau_{M})\cap\mathcal{B}_{\mathbb{R}^{D}}(p_{0},\tau_{0}).\label{eq:geometry.circlearc_inclusion_first}
	\end{equation}
	
	\item \textit{(Reverse inclusion)}	
	Let $x\in V\cap\partial\mathcal{B}_{\mathbb{R}^{D}}(z_{0},\tau_{M})\cap\mathcal{B}_{\mathbb{R}^{D}}(p_{0},\tau_{0})$,
	and fix $B_{i}=\mathcal{B}_{\mathbb{R}^{D}}\left(p_{i},\left\Vert p_{i}-q_{1}\right\Vert \right)$.
	Let $z_{0}'\in V$ be such that $\left\Vert z_{0}'-q_{1}\right\Vert =\left\Vert z_{0}'-q_{2}\right\Vert =\left\Vert p_{i}-q_{1}\right\Vert $
	and the line segment between $z_{0}'$ and $x$ intersects the
	line passing $q_{1}$ and $q_{2}$. Then $q_{1},q_{2},x,z_{0}'\in V$,
	and $x$ is not lying on the line passing $q_{1}$ and $q_{2}$. Hence
	from Lemma \ref{lem:geometry_plane_outpoint}, 
	\begin{equation}
	\left\Vert p_{i}-x\right\Vert \leq\left\Vert z_{0}'-x\right\Vert .\label{eq:geometry_circlearc_plane_outpoint}
	\end{equation}
	Since $x\in c_{z_{0}}(q_{1},q_{2})$ and $\left\Vert z_{0}'-q_{1}\right\Vert =\left\Vert z_{0}'-q_{2}\right\Vert <\tau_{M}=\left\Vert z_{0}-q_{1}\right\Vert =\left\Vert z_{0}-q_{2}\right\Vert $,
	Lemma \ref{lem:geometry_circledisk_intersection} yields
	\begin{equation}
	\left\Vert z_{0}'-x\right\Vert \leq\left\Vert z_{0}'-q_{1}\right\Vert .\label{eq:geometry_circlearc_circledisk_intersection}
	\end{equation}
	Hence \eqref{eq:geometry_circlearc_plane_outpoint} and \eqref{eq:geometry_circlearc_circledisk_intersection}
	gives the upper bound on $\left\Vert p_{i}-x\right\Vert $ as 
	\[
	\left\Vert p_{i}-x\right\Vert \leq\left\Vert z_{0}'-x\right\Vert \leq\left\Vert z_{0}'-q_{1}\right\Vert =\left\Vert p_{i}-q_{1}\right\Vert .
	\]
	Hence $x\in B_{i}$, and since choice of $x$ and $B_{i}$ were arbitrary,
	$V\cap\partial\mathcal{B}_{\mathbb{R}^{D}}(z_{0},\tau_{M})\cap\mathcal{B}_{\mathbb{R}^{D}}(p_{0},\tau_{0})\subset A_{\infty}$.
	But $\partial\mathcal{B}_{\mathbb{R}^{D}}(z_{0},\tau_{M})\cap\accentset{\circ}{\mathcal{B}}_{\mathbb{R}^{D}}(z_{0},\tau_{M})=\emptyset$,
	so that we get the desired inclusion
	\begin{equation}
	V\cap\partial\mathcal{B}_{\mathbb{R}^{D}}(z_{0},\tau_{M})\cap\mathcal{B}_{\mathbb{R}^{D}}(p_{0},\tau_{0})\subset A_{\infty}\setminus\accentset{\circ}{\mathcal{B}}_{\mathbb{R}^{D}}(z_{0},\tau_{M}).\label{eq:geometry.circlearc_inclusion_second}
	\end{equation}
\end{itemize}	
	Putting together \eqref{eq:geometry.circlearc_representation}, \eqref{eq:geometry.circlearc_inclusion_first},
	and \eqref{eq:geometry.circlearc_inclusion_second} we get
	\[
	A_{\infty}\setminus\accentset{\circ}{\mathcal{B}}_{\mathbb{R}^{D}}(z_{0},\tau_{M})=c_{z_{0}}(q_{1},q_{2}).
	\]
\end{proof}	
	
\begin{lem}
	\label{lem:geometry_reach_attain_global}
	Let $M\subset\mathbb{R}^{D}$ be a compact submanifold with reach
	$\tau_{M}>0$. If there exist $p\neq q\in M$ such that $\tau_{M}=\frac{\left\Vert q-p\right\Vert ^{2}}{2d(q-p,T_{p}M)}$,
	then there exists $z_{0}\in Med(M)$ with $d(z_{0},M)=\tau_{M}$.
\end{lem}

\begin{figure}[h!]
	\begin{center}
		\includegraphics{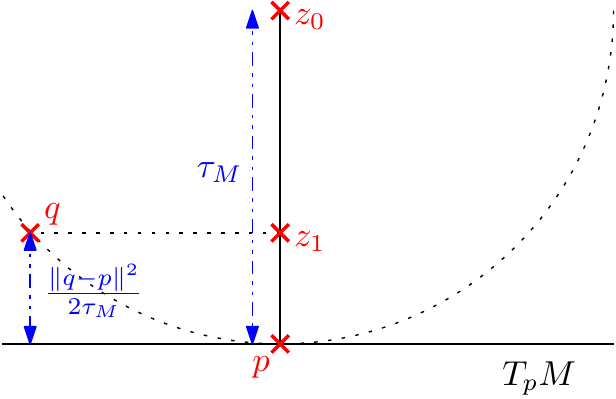}
	\end{center}
	\caption{
	Layout of the proof of Lemma \ref{lem:geometry_reach_attain_global}.
	}
	\label{fig:geometry_reach_attain_global}
\end{figure}

\begin{proof}[Proof of Lemma \ref{lem:geometry_reach_attain_global}]
	Write $z_{0}:=p+\tau_{M}\frac{\pi_{T_{p}M^{\perp}}(q-p)}{\left\Vert \pi_{T_{p}M^{\perp}}(q-p)\right\Vert }$.
	Clearly, $\left\Vert z_{0}-p\right\Vert =\tau_{M}$, and $z_{0}-p \in T_{p}M^{\perp}$, 
	so \cite[Theorem 4.8 (12)]{Federer1959} implies that for all $\lambda\in(0,1)$, $\pi_{M}(p+\lambda(z_{0}-p)) = p$, and hence $d(p+\lambda(z_{0}-p),M)=\left\Vert\lambda(z_{0}-p)\right\Vert=\lambda\tau_{M}$. Sending $\lambda\to 1$ yields that $d(z_{0},M)=\tau_{M}$.
	Let us show that $\left\Vert z_{0}-q\right\Vert =\tau_{M}$, which	will imply that $\left\Vert z_{0}-p\right\Vert = \left\Vert z_{0}-q\right\Vert = d(z_{0},M)=\tau_{M}$, and hence that $z_{0}\in Med(M)$,
	which will conclude the proof.
	
	Let $z_{1}:=p+\pi_{T_{p}M^{\perp}}(q-p)$ (see Figure \ref{fig:geometry_reach_attain_global}). Note that $z_{0}-z_{1}$
	and $q-z_{1}$ are simplified as 
	\begin{align*}
	z_{0}-z_{1} & =\left(\frac{\tau_{M}}{\left\Vert \pi_{T_{p}M^{\perp}}(q-p)\right\Vert }-1\right)\pi_{T_{p}M^{\perp}}(q-p),\\
	q-z_{1} & =(q-p)-\pi_{T_{p}M^{\perp}}(q-p)=\pi_{T_{p}M}(q-p).
	\end{align*}
	In particular, $z_{0}-z_{1}\perp q-z_{1}$, which yields
	\begin{align*}
	\left\Vert z_{0}-q\right\Vert^2  & =\left\Vert z_{0}-z_{1}\right\Vert ^{2}+\left\Vert q-z_{1}\right\Vert ^{2}
	\\
	&=
	\left(\tau_{M}-\left\Vert \pi_{T_{p}M^{\perp}}(q-p)\right\Vert \right)^{2}+\left\Vert \pi_{T_{p}M}(q-p)\right\Vert ^{2}
	.
	\end{align*}
	Noticing that
	\begin{align*}
	\left\Vert \pi_{T_{p}M^{\perp}}(q-p)\right\Vert  & =d(q-p,T_{p}M)=\frac{\left\Vert q-p\right\Vert ^{2}}{2\tau_{M}},
	\end{align*}
and	
	\begin{align*}
	\left\Vert \pi_{T_{p}M}(q-p)\right\Vert^2
	&=
	\left\Vert q-p\right\Vert ^{2}-\left\Vert \pi_{T_{p}M^{\perp}}(q-p)\right\Vert ^{2}
	=
	\left\Vert q-p\right\Vert^2 
	\left(
	1-\frac{\left\Vert q-p\right\Vert ^{2}}{4\tau_{M}^{2}}
	\right),
	\end{align*}
	we finally get
	\begin{align*}
	\left\Vert z_{0}-q\right\Vert^2
	&=
	\left(\tau_{M}-\frac{\left\Vert q-p\right\Vert ^{2}}{2\tau_{M}}\right)^{2}+\left\Vert q-p\right\Vert ^{2}\left(1-\frac{\left\Vert q-p\right\Vert ^{2}}{4\tau_{M}^{2}}\right)
	\\
	& =\tau_{M}^2.
	\end{align*}
\end{proof}

\begin{lem} \label{lem:geometry.maxcurvature_bound}
	Let $M\subset\mathbb{R}^{D}$ be a closed submanifold with reach $\tau_{M}>0$.
	Then for all $p,q\in M$ with $t_{0}:=d_{M}(p,q)\leq\tau_{M}/2$,
	\begin{equation*}
	\left\Vert \gamma_{p\to q}''(0)\right\Vert \leq\frac{2d(q-p,T_{p}M)}{\left\Vert q-p\right\Vert ^{2}}+\frac{2}{t_{0}^{2}}\left\Vert \int_{0}^{t_{0}}\int_{0}^{t}(\gamma_{p\to q}''(s)-\gamma_{p\to q}''(0))dsdt\right\Vert,
	\end{equation*}
	and 
	\begin{equation*}
	\left\Vert \gamma_{p\to q}''(0)\right\Vert \geq\frac{2d(q-p,T_{p}M)}{\left\Vert q-p\right\Vert ^{2}}-\frac{3\left\Vert q-p\right\Vert }{\tau_{M}^{2}}-\frac{2}{t_{0}^{2}}\left\Vert \int_{0}^{t_{0}}\int_{0}^{t}(\gamma_{p\to q}''(s)-\gamma_{p\to q}''(0))dsdt\right\Vert.	\end{equation*}
	In particular, when $M$ is $\mathcal{C}^{2}$, 
	\begin{equation*}
	\sup_{\substack{p\in M\\
			v\in{T_{p}M},\norm{v}=1
		}
	}\norm{\gamma_{p,v}''(0)}
	=
	\limsup_{\substack{q\to p \\q\in M}}\frac{2d(q-p,T_{p}M)}{\left\Vert q-p\right\Vert ^{2}}.\label{eq:geometry.maxcurvature_limit}
	\end{equation*}
	
\end{lem}

To prove Lemma \ref{lem:geometry.maxcurvature_bound} we need
the following straightforward result.

\begin{lem}\label{lem:geometry.decomposition_projection_bound}
	Let $U$ be a linear space and $u\in U$, $n\in U^{\perp}$. If $v=u+n+e$,
	then 
	\[
	\left|d(v,U)-\left\Vert n\right\Vert \right|\leq\left\Vert e\right\Vert .
	\]
\end{lem}

\begin{proof}[Proof of Lemma \ref{lem:geometry.maxcurvature_bound}]
	
	First note that from Proposition \ref{thm:geometry:injectivity_radius}
	(ii), $d_{M}(x,y)<\pi\tau_{M}$ ensures the existence and uniqueness
	of the geodesic $\gamma$. For short, let us write $\gamma=\gamma_{p\rightarrow q}$.
	
	The Taylor expansion of $\gamma$ at order two yields 
	\begin{align}
	q-p & =\gamma(t_{0})-\gamma(0)=t_{0}\gamma'(0)+\int_{0}^{t_{0}}\int_{0}^{t}\gamma''(s)dsdt\nonumber \\
	& =t_{0}\gamma'(0)+\frac{t_{0}^{2}}{2}\gamma''(0)+\int_{0}^{t_{0}}\int_{0}^{t}(\gamma''(s)-\gamma''(0))dsdt.\label{eq:taylor_geodesic}
	\end{align}
	Since $\gamma'(0)\in T_{p}M$ and $\gamma''(0)\in T_{p}M^{\perp}$,
	Lemma \ref{lem:geometry.decomposition_projection_bound} shows that
	\[
	\left|\frac{2d(q-p,T_{p}M)}{t_{0}^{2}}-\left\Vert \gamma''(0)\right\Vert \right|\leq\frac{2}{t_{0}^{2}}\left\Vert \int_{0}^{t_{0}}\int_{0}^{t}(\gamma''(s)-\gamma''(0))dsdt\right\Vert .
	\]
	Now, from $t_{0}=d_{M}(p,q)\geq\left\Vert q-p\right\Vert $, we derive
	the upper bound
	\begin{align*}
	\left\Vert \gamma''(0)\right\Vert  & \leq\frac{2d(q-p,T_{p}M)}{d_{M}(p,q)^{2}}+\frac{2}{t_{0}^{2}}\left\Vert \int_{0}^{t_{0}}\int_{0}^{t}(\gamma''(s)-\gamma''(0))dsdt\right\Vert \\
	& \leq\frac{2d(q-p,T_{p}M)}{\left\Vert q-p\right\Vert ^{2}}+\frac{2}{t_{0}^{2}}\left\Vert \int_{0}^{t_{0}}\int_{0}^{t}(\gamma''(s)-\gamma''(0))dsdt\right\Vert .
	\end{align*}
	For the lower bound, we apply \cite[Proposition 6.3]{NiyogiSW2008} to get
	\begin{align*}
	d_{M}(p,q)^{2} & \leq\tau_{M}^{2}\left(1-\sqrt{1-\frac{2\norm{q-p}}{\tau_{M}}}\right)^{2}\\
	& \leq\tau_{M}^{2}\frac{\left(\frac{\norm{q-p}}{\tau_{M}}\right)^{2}}{\left(1-\frac{2\norm{q-p}}{\tau_{M}}\right)^{3/2}}
	\leq\frac{\norm{q-p}^{2}}{1-3\frac{\norm{q-p}}{\tau_{M}}},
	\end{align*}
	or equivalently,
	\[
	\frac{1}{\norm{q-p}^{2}}-\frac{1}{d_{M}(p,q)^{2}}\leq\frac{3}{\tau_{M}\norm{q-p}}.
	\]
	As $d(q-p,T_{p}M)\leq\frac{\left\Vert q-p\right\Vert ^{2}}{2\tau_{M}}$ \eqref{eqn:reach_as_a_supremum_federer}, we finally derive 
	\begin{align*}
	\left\Vert \gamma''(0)\right\Vert  & \geq\frac{2d(q-p,T_{p}M)}{d_{M}(p,q)^{2}}-\frac{2}{t_{0}^{2}}\left\Vert \int_{0}^{t_{0}}\int_{0}^{t}(\gamma''(s)-\gamma''(0))dsdt\right\Vert \\
	& =\frac{2d(q-p,T_{p}M)}{\left\Vert q-p\right\Vert ^{2}}-2d(q-p,T_{p}M)\left(\frac{1}{\left\Vert q-p\right\Vert ^{2}}-\frac{1}{d_{M}(p,q)^{2}}\right)\\
	& \quad -\frac{2}{t_{0}^{2}}\left\Vert \int_{0}^{t_{0}}\int_{0}^{t}(\gamma''(s)-\gamma''(0))dsdt\right\Vert \\
	& \geq \frac{2d(q-p,T_{p}M)}{\left\Vert q-p\right\Vert ^{2}}-\frac{3\left\Vert q-p\right\Vert }{\tau_{M}^{2}}-\frac{2}{t_{0}^{2}}\left\Vert \int_{0}^{t_{0}}\int_{0}^{t}(\gamma''(s)-\gamma''(0))dsdt\right\Vert .
	\end{align*}
\end{proof}

\begin{proof}[Proof of Lemma \ref{thm:geometry.principalcurvature}]
For $r>0$, let $\Delta_{r}:=\left\{ (p,q)\in M^{2}|\left\Vert p-q\right\Vert <r\right\}$, and $\bar{\Delta} = \cap_{r>0} \Delta_r$ denote the diagonal of $M^2$.
Consider the map $\varphi: M^{2}\setminus \bar{\Delta} \to\mathbb{R}$ defined by $\varphi(p,q)={2d(q-p,T_{p}M)}/{\left\Vert q-p\right\Vert ^{2}}$.
By assumption, $d(z,M)>\tau_{M}$ for all $z\in Med(M)$. From Lemma \ref{lem:geometry_reach_attain_global}, this implies that for
all $p\neq q\in M$, $\varphi(p,q) < \tau_{M}^{-1}$. By compactness of $M^{2}\setminus\Delta_{r}$, this yields $\sup_{M^2\setminus \Delta_{r}} \varphi < \tau_M^{-1}.$ Hence, from the decomposition of \eqref{eqn:reach_as_a_supremum_federer} as
\begin{align*}
\frac{1}{\tau_{M}}
&=
\sup_{(p,q) \in M^2 \setminus \bar{\Delta}} \varphi(p,q)
=
\max
\left\{
\sup_{(p,q) \in M^2 \setminus \Delta_{r} } \varphi(p,q)
,
\sup_{(p,q) \in \Delta_{r} \setminus \bar{\Delta}} \varphi(p,q)
\right\}
,
\end{align*}
we get $\sup_{\Delta_{r} \setminus \bar{\Delta} } \varphi = \tau_M^{-1}$.
By letting $r>0$ go to zero and applying Lemma~\ref{lem:geometry.maxcurvature_bound}, this yields
\[
\sup_{\substack{p\in M\\
		v\in T_{p}M,\left\Vert v\right\Vert =1
	}
}\left\Vert \gamma_{p,v}''(0)\right\Vert =\lim_{r\to0}\sup_{(p,q)\in\Delta_{r}\setminus\bar{\Delta}}\varphi(p,q)
=
\frac{1}{\tau_M}.
\]
Finally, the unit tangent bundle $T^{(1)}M = \left\{(p,v), p \in M, v \in {T_p M}, \norm{v}=1 \right\}$ being compact, there exists $(q_{0},v_0) \in T^{(1)}M$ such that $\gamma_0 = \gamma_{q_{0},v_0}$ attains the supremum, i.e. $\norm{\gamma_0''(0)} = \tau_{M}^{-1}$, which concludes the proof.
\end{proof}

\section{Analysis of the Estimator}

\subsection{Global Case}
\label{sec:appendix:global_case}

\begin{proof}[Proof of Proposition \ref{thm:estimator.global_twopoints}]
The two left hand inequalities are direct consequences of Corollary \ref{thm:estimator.overestimate}, let us then focus on the third one.

We set $t$ to be equal to $\max\left\{ d_M(q_{1},x),d_M(q_{2},y)\right\}$, and $z_{1}:=x+ (q_{2}-q_{1})$. We have $\norm{z_{1}-x}=\norm{q_{2}-q_{1}}=2\tau_{M}$ and $\norm{y-q_{2}},\norm{q_{1}-x}\leq t$. Therefore,
from the definition of $\hat{\tau}$ in \eqref{eq:estimator.estimator} and the fact that the distance function to a linear space is $1$-Lipschitz, we get
\begin{align*}
\frac{1}{\hat{\tau}(\{x,y\})}
&\geq
\frac
{
	2d(y-x,T_{x}M)
}
{
	\norm{y-x}^2
}
\\
&=
\frac
{
	2d\left( (y-q_2) +(z_1-x) + (q_1-x),T_x M\right)
}
{
	\norm{(y-q_2) +(z_1-x) + (q_1-x)}^2
}
\\
& \geq
\frac{
	d(z_{1}-x,T_{x} M)-2t
}
{
	2(\tau_{M}+t)^{2}
}.
\end{align*}
Since $q_{1},q_{2}\in\mathcal{B}(z_{0},\tau_{M})$
and $\|q_{1}-q_{2}\|=2\tau_{M}$, 
$
z_{1}-x = q_{2}-q_{1} \in T_{q_{1}}M^\perp.
$
Furthermore, from \cite[Lemma 11]{Boissonnat18}, $\sin \angle (T_x M, T_{q_1} M) \leq t/\tau_M$ and hence
\begin{align*}
d(z_{1}-x,T_{x}M)
&\geq
d(z_1-x, T_{q_1} M) - \norm{z_1-x} \sin \angle (T_x M, T_{q_1} M)
\\
&\geq
d(q_2-q_1, T_{q_1} M) - \norm{q_2-q_1} \frac{t}{\tau_M}
\\
&=
2\tau_M \left( 1 - \frac{t}{\tau_M} \right).
\end{align*}
Combining the two previous bounds finally yields the announced result
\begin{align*}
\frac{1}{\tau_{M}}
-
\frac{1}{\hat{\tau}(\{x,y\})} 
& \leq
\frac{1}{\tau_{M}}
-
\frac{d(z_{1}-x,T_{x} M)-2t}{2(\tau_{M}+t)^{2}}
\\
& \leq
\frac{1}{\tau_{M}}
\left(
1-
\frac
{
1-2 {t}/{\tau_M}
}
{
\left(1+{t}/{\tau_{M}}\right)^{2}
}
\right)
\\
& \leq
\frac{4}{\tau_{M}^{2}}t,
\end{align*}
where the last inequality follows from the concavity of $[0,1] \ni u \mapsto 1 - \frac{1-2u}{(1+u)^2}$.
\end{proof}

\begin{proof}[Proof of Proposition \ref{thm:estimator.global_risk}]
Let $Q$ be the distribution on $\R^D$ associated to $P$.
Let $s<\frac{1}{\tau_{M}}$ and $t=\frac{\tau_{M}^2}{4}s \leq \tau_M/4$.
Write $\omega_{d}:=\mathcal{H}^d(\mathcal{B}_{\mathbb{R}^{d}}(0,1))$ for the volume of the $d$-dimensional unit ball. From Proposition \ref{thm:geometry:injectivity_radius} (v), for all $q \in M$,
	\begin{align*}
	Q
	\left(
	\mathcal{B}_M(q,t)
	\right)
	&\geq
	f_{min} \mathcal{H}^d\left(\mathcal{B}_M(q,t)\right)
	\\
	&\geq
	\omega_{d} f_{min} 
	\left(1-\left(\frac{t}{6\tau_{M}}\right)^{2}\right)^{d}
	t^{d}
	\\
	&\geq
	\omega_{d} f_{min} 
	\left(\frac{575}{576}\right)^{d}t^{d}.
	\end{align*}
Moreover, Proposition \ref{thm:estimator.global_twopoints} asserts that $\left| \frac{1}{\tau_M}-\frac{1}{\hat{\tau}(\X_n)} \right|>s$ implies that either $\mathcal{B}_{M}(q_{1},t)\cap\X_n =\emptyset$ or $\mathcal{B}_{M}(q_{2},t)\cap\X_n =\emptyset$.
Hence,
\begin{align*}
\P
\left(
\left| \frac{1}{\tau_M}-\frac{1}{\hat{\tau}(\X_n)} \right|>s
\right) 
& \leq 
\P
\left(
\mathcal{B}_{M}(q_{1},t)\cap\X_n=\emptyset
\right)
+
\P
\left(
\mathcal{B}_{M}(q_{2},t)\cap\X_n=\emptyset
\right)
\\
& \leq
2\left(1-\omega_{d} f_{min} \left(\frac{575}{576}\right)^{d} t^{d}\right)^{n}
\\
& \leq
2\exp\left(-n\omega_{d} f_{min} \left(\frac{575}{2304}\right)^{d}\tau_{M}^{2d}s^{d}\right).
\end{align*}
Integrating the above bound gives
\begin{align*}
\mathbb{E}_{P^{n}}
\left[
\left|
\frac{1}{\tau_M}-\frac{1}{\hat{\tau}(\X_n)} 
\right|^{p}
\right] 
& =
\int_{0}^{\frac{1}{\tau_{M}^{p}}}
\P
\left(
\left|
\frac{1}{\tau_M}-\frac{1}{\hat{\tau}(\X_n)}
\right|^{p}
>s
\right) d s
\\
& \leq
2\int_{0}^{\infty}
\exp\left(-n\omega_{d} f_{min} \left(\frac{575}{2304}\right)^{d}\tau_{M}^{2d}s^{\frac{d}{p}}\right)ds
\\
& =
\frac{2\left(\frac{2304}{575}\right)^{\frac{p}{d}}}{(n\omega_{d} f_{min})^{\frac{p}{d}}\tau_{M}^{2p}}\int_{0}^{\infty}x^{\frac{p}{d}-1}e^{-x}dx
\\
& 
:=
C_{\tau_{M},f_{\min},d,p}n^{-\frac{p}{d}}.
\end{align*}
where $C_{\tau_{M},f_{\min},d,p}$ depends only on $\tau_{M}$, $f_{min}$, $d$, $p$, and is a decreasing function of
$\tau_{M}$ when the other parameters are fixed. 
\end{proof}

\subsection{Local Case}
\label{sec:appendix:local_case}

\begin{proof}[Proof of Lemma \ref{lem:estimator.local_twopoints}]
First note that from Proposition \ref{thm:geometry:injectivity_radius} (ii), $d_{M}(x,y)<\pi\tau_{M}$ ensures the existence and uniqueness of the geodesic $\gamma_{x \rightarrow y}$.
The two left hand inequalities are direct consequences of Corollary \ref{thm:estimator.overestimate}. Let us then focus on the third one.
We write $t_{0}=d_{M}(x,y)$ and $\gamma = \gamma_{x \rightarrow y}$ for short.
By the definition~\eqref{eq:estimator.estimator} of $\hat{\tau}$,
\begin{equation}
\frac{1}{\hat{\tau}(\{x,y\})}
\geq
\frac{2d(y-x,T_{x} M)}{\norm{y-x}^2}.
\label{eq:estimator.local_twopoints_geodesic_approx}
\end{equation}
Furthermore, from Lemma \ref{lem:geometry.maxcurvature_bound},
\begin{equation}
\frac{2d(y-x,T_{x}M)}{\left\Vert y-x\right\Vert ^{2}}\geq\left\Vert \gamma''(0)\right\Vert -\frac{2}{t_{0}^{2}}\left\Vert \int_{0}^{t_{0}}\int_{0}^{t}(\gamma''(s)-\gamma''(0))dsdt\right\Vert .
\label{eq:estimator.local_twopoints_lowerbound_pre}
\end{equation}
But by definition of $\mathcal{M}_{\tau_{\min},L}^{d,D} \ni M$ (Definition \ref{def:geometric_model}), the geodesic $\gamma$ satisfies $\norm{\gamma''(s)-\gamma''(0)}\leq L|s|$, so that
\begin{align}
\frac{2}{t_{0}^{2}}\int_{0}^{t_{0}}\int_{0}^{t}\left\Vert \gamma''(s)-\gamma''(0)\right\Vert dsdt
&\leq
\frac{2}{t_{0}^{2}}\int_{0}^{t_{0}}\int_{0}^{t}L|s|dsdt
=
\frac{1}{3}Lt_{0}. \label{eq:estimator.local_twopoints_lowerbound_final}
\end{align}
Combining \eqref{eq:estimator.local_twopoints_geodesic_approx}, \eqref{eq:estimator.local_twopoints_lowerbound_pre} and \eqref{eq:estimator.local_twopoints_lowerbound_final} gives the announced inequality.
\end{proof}

To prove Lemma \ref{lem:estimator.local_butterfly}, we will use the following lemma on bilinear maps.

\begin{lem} \label{lem:estimator_bilinear}
	
	Let $\left(V,\left\langle \cdot,\cdot\right\rangle \right)$ and $\left(W,\left\langle \cdot,\cdot\right\rangle \right)$ be
	Hilbert spaces. Let $B:V\times V\to W$ be a continuous
	 bilinear map, and
	write 
	\[
	\lambda_{\max}:=\sup_{\substack{v\in V\\
			\left\Vert v\right\Vert =1
		}
	}\left\Vert B(v,v)\right\Vert .
	\]
	Then for all unit vectors $v,w\in V$,
	
	\begin{enumerate}[leftmargin=*]
		\item[(i)]  $\left\Vert B(w,w)-2\left\langle v,w\right\rangle ^{2}B(v,v)\right\Vert \leq(3-2\left\langle v,w\right\rangle ^{2})\lambda_{\max}$ .
		
		\item[(ii)]  If $v\in V$ satisfies that for all $\tilde{v}\perp v$, $\left\langle B(v,v),B(v,\tilde{v})+B(\tilde{v},v)\right\rangle =0$,
		then 
		\[
		\left\Vert B(w,w)\right\Vert \geq\left\langle v,w\right\rangle ^{2}\left\Vert B(v,v)\right\Vert -(1-\left\langle v,w\right\rangle ^{2})\lambda_{\max}.
		\]
		In particular, this holds whenever $\norm{v} = 1$ with $\left\Vert B(v,v)\right\Vert =\lambda_{\max}$.
	\end{enumerate}
\end{lem}

\begin{proof}[Proof of Lemma \ref{lem:estimator_bilinear}]
	
	Let $\theta=\arccos(\left\langle v,w\right\rangle )\in\left[0,\pi\right]$,
	and write $w=\cos\theta v+\sin\theta v^{\perp}$ for some unit vector
	$v^{\perp}\in V$ with $v^{\perp}\perp v$. Then $B(w,w)$ can be
	expanded as 
	\begin{equation}
	B(w,w)=\cos^{2}\theta B(v,v)+\cos\theta\sin\theta(B(v,v^{\perp})+B(v^{\perp},v))+\sin^{2}\theta B(v^{\perp},v^{\perp}).\label{eq:estimator_bilinear_expand_original}
	\end{equation}
	
	\begin{enumerate}[leftmargin=*]
		\item[(i)]  Consider $\bar{w}:=-\cos\theta v+\sin\theta v^{\perp}\in V$. Then
		$\bar{w}$ is a unit vector, and $B(\bar{w},\bar{w})$ can be similarly
		expanded as 
		\begin{equation}
		B(\bar{w},\bar{w})=\cos^{2}\theta B(v,v)-\cos\theta\sin\theta(B(v,v^{\perp})+B(v^{\perp},v))+\sin^{2}\theta B(v^{\perp},v^{\perp}),\label{eq:estimator_bilinear_expand_rotated}
		\end{equation}
		and hence summing up \eqref{eq:estimator_bilinear_expand_original}
		and \eqref{eq:estimator_bilinear_expand_rotated} gives 
		\begin{align*}
		B(w,w) + B(\bar{w},\bar{w}) & =2\cos^{2}\theta B(v,v)+2\sin^{2}\theta B(v^{\perp},v^{\perp}).
		\end{align*}
		As $\left\Vert B(v^{\perp},v^{\perp})\right\Vert$ and  $ \left\Vert B(\bar{w},\bar{w})\right\Vert$ are upper bounded by $\lambda_{\max}$,	this yields 
		\begin{align*}
		\left\Vert B(w,w)-2\cos^{2}\theta B(v,v)\right\Vert  & =\left\Vert 2\sin^{2}\theta B(v^{\perp},v^{\perp})-B(\bar{w},\bar{w})\right\Vert \\
		& \leq(1+2\sin^{2}\theta)\lambda_{\max}
		\\
		& =(3-2\cos^{2}\theta)\lambda_{\max},
		\end{align*}
		which is the announced bound.
		
		\item[(ii)]  From \eqref{eq:estimator_bilinear_expand_original}, $\left\Vert B(w,w)\right\Vert $
		can be lower bounded as 
		\begin{align}
		& \left\Vert B(w,w)\right\Vert \nonumber \\
		& \geq\left\Vert \cos^{2}\theta B(v,v)+\cos\theta\sin\theta(B(v,v^{\perp})+B(v^{\perp},v))\right\Vert -\sin^{2}\theta\left\Vert B(v^{\perp},v^{\perp})\right\Vert .\label{eq:estimator_bilinear_orthogonal_lower_bound}
		\end{align}
		But since $\left\langle B(v,v),B(v,v^{\perp})+B(v^{\perp},v)\right\rangle =0$, Pythagoras's theorem yields  
		\begin{align*}
		& \left\Vert \cos^{2}\theta B(v,v)+\cos\theta\sin\theta(B(v,v^{\perp})+B(v^{\perp},v))\right\Vert \\
		& =\sqrt{\cos^{4}\theta\left\Vert (B(v,v)\right\Vert ^{2}+\cos^{2}\theta\sin^{2}\theta\left\Vert (B(v,v^{\perp})+B(v^{\perp},v)\right\Vert ^{2}}\\
		& \geq\cos^{2}\theta\left\Vert (B(v,v)\right\Vert .
		\end{align*}
		Applying this and $\left\Vert B(v^{\perp},v^{\perp})\right\Vert \leq\lambda_{\max}$
		to \eqref{eq:estimator_bilinear_orthogonal_lower_bound} gives the final bound
		\[
		\left\Vert B(w,w)\right\Vert \geq\cos^{2}\theta\left\Vert B(v,v)\right\Vert -\sin^{2}\theta\lambda_{\max}.
		\]
		We now show the last claim, namely that $\norm{v} = 1$ and $\left\Vert B(v,v)\right\Vert =\lambda_{\max}$ are sufficient conditions for 
		\begin{equation}
		\left\langle B(v,v),B(v,\tilde{v})+B(\tilde{v},v)\right\rangle =0\text{ for all }\tilde{v}\perp v.\label{eq:estimator_bilinear_condition_orthogonal}
		\end{equation}
		For this aim, we take such a $v\in V$ and we consider $h:V\to\mathbb{R}$ defined by $h(u)=\left\Vert B(u,u)\right\Vert ^{2}$ and $g:V\to\mathbb{R}$ defined by $g(u)=\left\Vert u\right\Vert ^{2}-1$. Then $v$ is a solution of the optimization problem: 
		\begin{align*}
		& \text{maximize }h(u)\\
		& \text{subject to }g(u)=0.
		\end{align*}
		Since $h$ and $g$ are continuously differentiable, the Lagrange
		multiplier theorem asserts that their Fr{\' e}chet derivatives at
		$v$ satisfy $\ker d_v g \subset\ker d_v h$. 
		As $d_v h (u) =2\left\langle B(v,v),B(v,u)+B(u,v)\right\rangle$ and $d_v g(u) = 2\left\langle v,u\right\rangle $, this rewrites exactly as the claim \eqref{eq:estimator_bilinear_condition_orthogonal}.

	\end{enumerate}
	
\end{proof}

\begin{cor} \label{cor:estimator_curvatureangle} Let $M \subset \R^D$ be a $\mathcal{C}^2$-submanifold	and $p\in M$. Let $v_{0},v_{1}\in T_{p}M$ be unit tangent vectors, and let $\theta=\angle(v_{0},v_{1})$. 
Let $\gamma_{p,v}$ be the arc length parametrized geodesic starting from $p$ with velocity $v$, and write $\gamma_{i}=\gamma_{p,v_{i}}$ for $i=0,1$. Let $\kappa_{p}=\max_{v\in\mathcal{B}_{T_{p}M}(0,1)}\left\Vert \gamma_{p,v}''(0)\right\Vert $.
	Then, 
	\begin{enumerate}[leftmargin=*]
	\item[(i)]
		$\left\Vert \gamma_{1}''(0)\right\Vert \geq 2\left\Vert \gamma_{0}''(0)\right\Vert \cos^{2}\theta
		-
		\kappa_{p}
		(1+2\sin^{2}\theta)
		.$
		
	\item[(ii)] 
	If $v_{0}$ is a direction of maximum directional curvature,
	i.e. $\left\Vert \gamma_{0}''(0)\right\Vert =\kappa_p$, then
	$\left\Vert \gamma_{1}''(0)\right\Vert \geq\left\Vert \gamma_{0}''(0)\right\Vert -2\kappa_{p} \sin^{2}\theta.$
	\end{enumerate}
\end{cor}

\begin{proof}[Proof of Corollary \ref{cor:estimator_curvatureangle}]

	Consider the symmetric bilinear map $B:T_{p}M\times T_{p}M\to T_{p}M^{\perp}$
		given by the hessian of the exponential map $B(v,w):=d_{0}^{2}\exp_{p}(v,w).$
		In particular, for all $v\in T_{p}M$, $\gamma_{p,v}''(0)=B(v,v)$ and $\sup_{v\in V,\left\Vert v\right\Vert =1}\left\Vert B(v,v)\right\Vert =\kappa_{p}$.
		This allows us to tackle the two points of the result.
	
	\begin{enumerate}[leftmargin=*]
	\item[(i)]
		Applying Lemma \ref{lem:estimator_bilinear} (i) to $B$ with $v = v_0$ and $w = v_1$ yields
		\begin{align*}
		\left\Vert \gamma_{1}''(0)\right\Vert  & \geq\left\Vert 2\cos^{2}\theta\gamma_{0}''(0)\right\Vert -\left\Vert 2\cos^{2}\theta\gamma_{0}''(0)-\gamma_{1}''(0)\right\Vert \\
		& \geq
		2\left\Vert \gamma_{0}''(0)\right\Vert \cos^{2}\theta
		-
		\kappa_{p}
		(1+2\sin^{2}\theta)		
		\end{align*}
	
	\item[(ii)]	

	Since $v_{0}$ gives the maximal directional curvature, applying Lemma \ref{lem:estimator_bilinear} (ii) to $B$, $v = v_0$ and $w = v_1$ precisely yields
	$\left\Vert \gamma_{1}''(0)\right\Vert \geq\left\Vert \gamma_{0}''(0)\right\Vert - 2\kappa_p \sin^{2}\theta.$
	\end{enumerate}
\end{proof}

For a triangle in a Euclidean space, the sum of any two angles is upper bounded by $\pi$. The same property holds for a geodesic
triangle on a manifold if its side lengths are not too large compared
to its reach, which is formalized in the following Lemma \ref{lem:estimator_triangle_angle}.

\begin{lem}
	
	\label{lem:estimator_triangle_angle}
	
	Let $M\subset\mathbb{R}^{D}$ be a closed submanifold with reach $\tau_{M}>0$, and $x,y,z\in M$ be three distinct points. 
	Consider the geodesic triangle with vertices $x,y,z$, that is, the triangle formed by $\gamma_{x\to y}$, $\gamma_{y\to z}$, $\gamma_{z\to x}$. 
	
	If at least two of the side lengths of the triangle are strictly less than $\frac{\pi\tau_{M}}{2}$, then the sum of any two of its angles is less than or equal to $\pi$.
	
\end{lem}

\begin{proof}[Proof of Lemma \ref{lem:estimator_triangle_angle}]
	
	Without loss of generality, suppose $d_{M}(x,y)$ is the longest side
	length: $d_{M}(y,z),d_{M}(z,x)\leq d_{M}(x,y)$. Then $d_{M}(y,z),d_{M}(z,x)\in\left(0,\frac{\pi\tau_{M}}{2}\right)$, so that $d_{M}(x,y)\in(0,\pi\tau_{M})$ by triangle inequality.
	
	Let $\mathcal{S}_{\tau_{M}}^{2}$ be a $d$-dimensional sphere of
	radius $\tau_{M}$. In what follows, for short, $\angle abc$ stands
	for $\angle(\gamma'_{b\to a}(0),\gamma'_{b\to c}(0))$. Let $\bar{x},\bar{y},\bar{z}\in\mathcal{S}_{\tau_{M}}^{2}$
	be such that $d_{\mathcal{S}_{\tau_{M}}^{2}}(\bar{x},\bar{y})=d_{M}(x,y)$,
	$d_{\mathcal{S}_{\tau_{M}}^{2}}(\bar{y},\bar{z})=d_{M}(y,z)$, and
	$d_{\mathcal{S}_{\tau_{M}}^{2}}(\bar{z},\bar{x})=d_{M}(z,x)$. From
	Proposition~\ref{thm:geometry:injectivity_radius} (ii) and the fact that $d_{M}(x,y)+d_{M}(y,z)+d_{M}(z,x)<2\pi\tau_{M}$,
	Toponogov's comparison theorem \cite[Section 4]{Karcher89}
	yields $\angle xyz\leq\angle\bar{x}\bar{y}\bar{z}$, $\angle yzx\leq\angle\bar{y}\bar{z}\bar{x}$,
	and $\angle xzy\leq\angle\bar{x}\bar{z}\bar{y}$. 
	Furthermore, the spherical law of cosines \cite[Proposition 18.6.8]{Berger87} together with $d_{M}(y,z),d_{M}(z,x)\in\left(0,\frac{\pi\tau_{M}}{2}\right)$,
	$d_{M}(x,y)\in(0,\pi\tau_{M})$, and the fact that $\cos(\cdot)$ is decreasing on $[0,\pi]$ imply 
	\[
	\cos\left(\angle\bar{z}\bar{x}\bar{y}\right)=\frac{\cos\left(\frac{d_{M}(y,z)}{\tau_{M}}\right)-\cos\left(\frac{d_{M}(z,x)}{\tau_{M}}\right)\cos\left(\frac{d_{M}(x,y)}{\tau_{M}}\right)}{\sin\left(\frac{d_{M}(z,x)}{\tau_{M}}\right)\sin\left(\frac{d_{M}(x,y)}{\tau_{M}}\right)}\geq0,
	\]
	so that $\angle zxy\leq\angle\bar{z}\bar{x}\bar{y}\leq\frac{\pi}{2}$.
	Symmetrically, we also have $\angle xyz\leq\frac{\pi}{2}$.
	
	If $\angle yzx\leq\frac{\pi}{2}$ also holds, then the final result is trivial, so from now on we will assume that $\angle yzx\geq\frac{\pi}{2}$.
	
	Thus, $\sin(\angle\bar{y}\bar{z}\bar{x})\leq\sin(\angle yzx)$, so applying the spherical law of sines and cosines \cite[Proposition 18.6.8]{Berger87},
	$d_{M}(y,z),d_{M}(z,x)\in\left(0,\frac{\pi\tau_{M}}{2}\right)$, and
	$\angle\bar{y}\bar{z}\bar{x}\in\left[\frac{\pi}{2},\pi\right]$ yield
	\begin{align*}
	& \sin(\angle zxy)\leq\sin(\angle\bar{z}\bar{x}\bar{y})
	\\
	& =\frac{\sin\left(\frac{d_{M}(y,z)}{\tau_{M}}\right)\sin(\angle\bar{y}\bar{z}\bar{x})}{\sqrt{1-\left(\cos\left(\frac{d_{M}(z,x)}{\tau_{M}}\right)\cos\left(\frac{d_{M}(y,z)}{\tau_{M}}\right)+\sin\left(\frac{d_{M}(z,x)}{\tau_{M}}\right)\sin\left(\frac{d_{M}(y,z)}{\tau_{M}}\right)\cos(\angle\bar{y}\bar{z}\bar{x})\right)^{2}}}
	\\
	& \leq\frac{\sin\left(\frac{d_{M}(y,z)}{\tau_{M}}\right)\sin(\angle\bar{y}\bar{z}\bar{x})}{\sqrt{1-\cos^{2}\left(\frac{d_{M}(z,x)}{\tau_{M}}\right)\cos^{2}\left(\frac{d_{M}(y,z)}{\tau_{M}}\right)}}
	\leq\sin(\angle\bar{y}\bar{z}\bar{x})
	\leq
	\sin(\angle yzx).
	\end{align*}
	This last bound together with $\angle zxy\leq\frac{\pi}{2}\leq\angle yzx$ yield $\angle zxy+\angle yzx\leq\pi$. 
	Symmetrically, we also have $\angle xxz+\angle xzy\leq\pi$. 
	Hence, the sum of any two angles is less than or equal to $\pi$.
	
\end{proof}

We are now in position to prove Lemma \ref{lem:estimator.local_butterfly}.

\begin{proof}[Proof of Lemma \ref{lem:estimator.local_butterfly}]
For short, in what follows, we let $t_{x}:=d_{M}(q_{0},x)$, $t_{y}:=d_{M}(q_{0},y)$, and 
$\theta:=\angle(\gamma_{x\to y}'(0),\gamma_{x\to q_{0}}'(0))=\pi - \angle(\gamma_{x\to y}'(0),\gamma_{q_{0}\to x}'(t_{x}))$ (see Figure \ref{fig:estimator.local_butterfly}).
From Corollary \ref{cor:estimator_curvatureangle} (i),
\begin{equation}
\left\Vert \gamma_{x\to y}''(0)\right\Vert \geq(2-2\sin^{2}\theta)\left\Vert \gamma_{q_{0}\to x}''(t_{x})\right\Vert -(1+2\sin^{2}\theta)\kappa_{x}.\label{eq:estimator_curvature_two_points}
\end{equation}
We now focus on the term $\left\Vert \gamma_{q_{0}\to x}''(t_{x})\right\Vert$. 
Since the direction $\gamma_{0}'(0)$ maximizes the directional curvature at $q_{0}$ and $\theta_{x}=\angle(\gamma_{0}'(0),\gamma_{q_{0}\to x}'(0))$, Corollary \ref{cor:estimator_curvatureangle} (ii)
yields
\[
\left\Vert \gamma_{q_{0}\to x}''(0)\right\Vert \geq(1-2\sin^{2}\theta_{x})\kappa_{q_{0}},
\]
and since $\gamma''_{q_0 \rightarrow x}$ is $L$-Lipschitz,
\begin{align}
\left\Vert \gamma_{q_{0}\to x}''(t_{x})\right\Vert  & \geq\left\Vert \gamma_{q_{0}\to x}''(0)\right\Vert -\left\Vert \gamma_{q_{0}\to x}''(t_{x})-\gamma_{q_{0}\to x}''(0)\right\Vert \nonumber \\
& \geq(1-2\sin^{2}\theta_{x})\kappa_{q_{0}}-Lt_{x}.\label{eq:estimator_curvature_one_point}
\end{align}
\begin{figure}[h]
	\centering
	\includegraphics[width=0.6\textwidth]{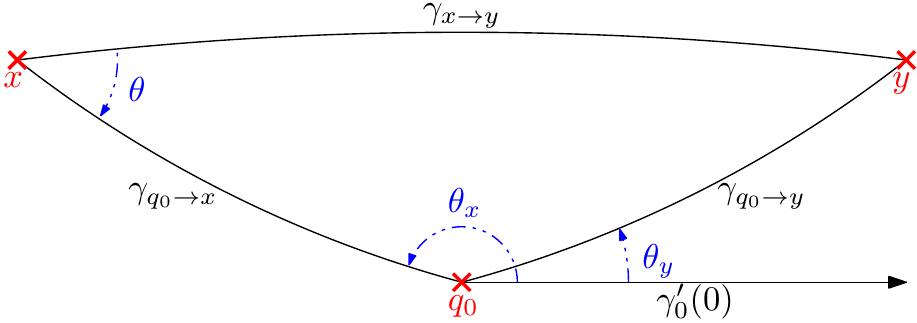}
	\caption{{Layout of Lemma \ref{lem:estimator.local_butterfly}.}}
	\label{fig:estimator.local_butterfly}
\end{figure}
Now, consider the geodesic triangle with vertices $x,y,q_{0}$, that is, the triangle formed by $\gamma_{x\to y}$, $\gamma_{q_{0}\to x}$, $\gamma_{q_{0}\to y}$,
as in Figure \ref{fig:estimator.local_butterfly}. Then Lemma \ref{lem:estimator_triangle_angle}
implies that 
$
\theta+|\theta_{x}-\theta_{y}|\leq\pi.$
Combined with the assumption that $|\theta_{x}-\theta_{y}|\geq\frac{\pi}{2}$, this yields
\begin{equation}
\sin\theta \leq\sin(|\theta_{x}-\theta_{y}|).\label{eq:estimator_sinbound}
\end{equation}
Putting together \eqref{eq:estimator_curvature_two_points}, 
\eqref{eq:estimator_curvature_one_point} and \eqref{eq:estimator_sinbound} gives the final bound
\begin{align*}
& \left\Vert \gamma_{x\to y}''(0)\right\Vert \\
& \geq2(1-\sin^{2}(|\theta_{x}-\theta_{y}|))\left((1-2\sin^{2}\theta_{x})\kappa_{q_{0}}-Lt_{x}\right)-(1+2\sin^{2}(|\theta_{x}-\theta_{y}|))\kappa_{x}\\
& =2\kappa_{q_{0}}-\kappa_{x}(1+2\sin^{2}(|\theta_{x}-\theta_{y}|))-2Lt_{x}\cos^{2}(|\theta_{x}-\theta_{y}|)\\
& \quad-2\kappa_{q_{0}}\left(\sin^{2}(|\theta_{x}-\theta_{y}|))+2\sin^{2}\theta_{x}-\sin^{2}\theta_{x}\sin^{2}(|\theta_{x}-\theta_{y}|)\right)\\
& \geq\kappa_{q_{0}}-(\kappa_{x}-\kappa_{q_{0}})-2Lt_{x}-(2\kappa_{x}+6\kappa_{q_{0}})\sin^{2}(|\theta_{x}-\theta_{y}|).
\end{align*}

\end{proof}

\begin{proof}[Proof of Proposition \ref{thm:estimator.local_risk}]
	
	In what follows, we let $t_{0}\leq\frac{\tau_{\min}}{10}$, 
\begin{align*}
B_{1}&:=\exp_{q_{0}}\Bigl(\Bigl\{ v\in T_{q_{0}}M:\,\left\Vert v\right\Vert \leq t_{0},\,\angle(\gamma_{0}'(0),v)\leq\sqrt{\frac{t_{0}}{\tau_{\min}}}\Bigr\}\Bigr),
\\
B_{2}&:=\exp_{q_{0}}\Bigl(\Bigl\{ v\in T_{q_{0}}M:\,\left\Vert v\right\Vert \leq t_{0},\,\angle(\gamma_{0}'(0),v)\geq\pi-\sqrt{\frac{t_{0}}{\tau_{\min}}}\Bigr\}\Bigr)
,
\end{align*}
and $B_{0}:=B_{1}\cup B_{2}$ (see Figure \ref{fig:estimator.cone}).
\begin{figure}
	\centering
	\includegraphics[width=0.5\textwidth]{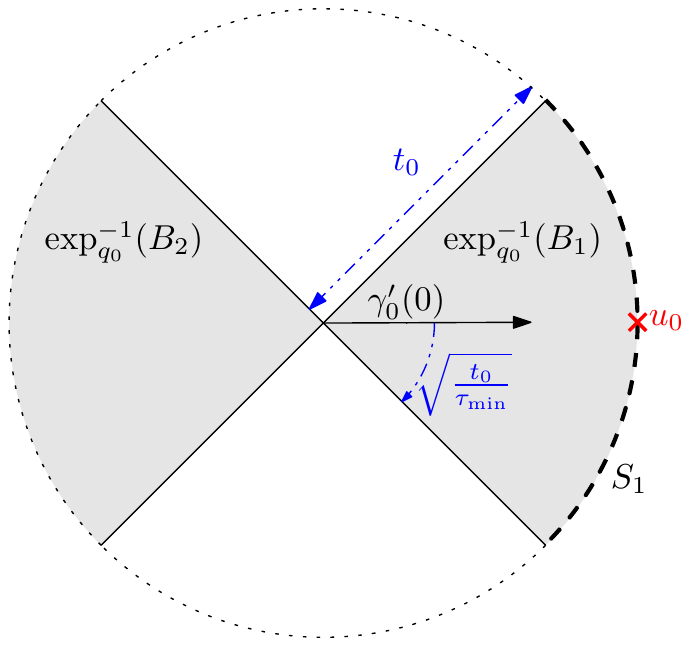}
	\caption{Layout of the proof of  Proposition \ref{thm:estimator.local_risk}.}\label{fig:estimator.cone}
\end{figure}
Let $\mathbb{X}\subset M$, and $x,y\in\mathbb{X}$
	be such that $x\in B_{1}$, $y\in B_{2}$.
	Writing $\theta_{x}:=\angle(\gamma_{0}'(0),\gamma_{q_{0}\to x}'(0))$
	and $\theta_{y}:=\angle(\gamma_{0}'(0),\gamma_{q_{0}\to y}'(0))$,
	then $\theta_{x}\leq\sqrt{\frac{t_{0}}{\tau_{\min}}}\leq\frac{\pi}{4}$
	and $\theta_{y}\geq\pi-\sqrt{\frac{t_{0}}{\tau_{\min}}}\geq\frac{3\pi}{4}$. Also, $d_{M}(q_{0},x)\leq t_{0}$ and $d_{M}(x,y)\leq2t_{0}$, so Proposition \ref{thm:estimator.local_twopoints} rewrites as
	\begin{align*}
	0 \leq\frac{1}{\tau_{M}}-\frac{1}{\hat{\tau}(\mathbb{X})}
	&\leq
	\frac{8\sin^{2}(|\theta_{x}-\theta_{y}|)}{\tau_{M}}+L\left(\frac{1}{3}d_{M}(x,y)+2d_{M}(q_{0},x)\right)\\
	& \leq\left(\frac{16}{\tau_{\min}\tau_{M}}+\frac{8L}{3}\right)t_{0}.
	\end{align*}
	A symmetric argument also applies when $x\in B_{2}$ and $y\in B_{1}$.
	Now, for any $s<\frac{1}{\tau_{M}}$, let
	$t_{0}(s):=\left(\frac{16}{\tau_{\min}^{2}}+\frac{8L}{3}\right)^{-1}s<\frac{\tau_{\min}}{10}$.
The above argument implies that if $\left|\frac{1}{\tau_{M}} - \frac{1}{\hat{\tau}(\mathbb{X})}\right|>s$, 	then for any $x,y\in\mathbb{X} \cap B_{0}$, one has either $x,y\in B_{1}$ or $x,y\in B_{2}$. Hence,
	\begin{align}
	& \mathbb{P}\left(\left|\frac{1}{\tau_{M}}-\frac{1}{\hat{\tau}(\X_n)}\right|>s\right)\nonumber \\
	& \leq\sum_{m=0}^{n} {{n}\choose{m}} \Bigl\{\mathbb{P}\left(X_{1},\ldots,X_{m}\in M\backslash B_{0},X_{m+1},\ldots,X_{n}\in B_{1}\right)\nonumber \\
	& \qquad\qquad\qquad\quad+\mathbb{P}\left(X_{1},\ldots,X_{m}\in M\backslash B_{0},X_{m+1},\ldots,X_{n}\in B_{2}\right)\Bigr\}\nonumber \\
	& =\sum_{m=0}^{n} {{n}\choose{m}}\left\{(1-Q(B_{0}))^{m}Q(B_{1})^{n-m}+(1-Q(B_{0}))^{m}Q(B_{2})^{n-m}\right\}\nonumber \\
	&\leq
	(1-Q(B_{2}))^{n}+(1-Q(B_{1}))^{n}
	.\label{eq:estimator_expansionbound}
	\end{align}
We now derive lower bounds for $Q(B_{1})$ and $Q(B_{2})$. 
	For this purpose, let $S_{1} := \exp_{q_{0}}^{-1}(B_{1})\cap\partial\mathcal{B}_{T_{q_{0}}M}(0,t_{0})$ (see Figure \ref{fig:estimator.cone}). Then $\exp_{q_{0}}^{-1}(B_{1})\subset \mathcal{B}_{T_{q_{0}}M}(0,t_{0})$ is a cone satisfying 
	\[
	\frac{\mathcal{H}^{d}\left(\exp_{q_{0}}^{-1}(B_{1})\right)}{\mathcal{H}^{d}\left(\mathcal{B}_{T_{q_{0}}M}(0,t_{0})\right)}=\frac{\mathcal{H}^{d-1}\left(S_{1}\right)}{\mathcal{H}^{d-1}\left(\partial\mathcal{B}_{T_{q_{0}}M}(0,t_{0})\right)}.
	\]
	Let $\omega_{d}:=\mathcal{H}^{d}(\mathcal{B}_{\mathbb{R}^{d}}(0,1))$
	and $\sigma_{d}:=\mathcal{H}^{d}(\partial\mathcal{B}_{\mathbb{R}^{d+1}}(0,1))$
	denote the volumes of the $d$-dimensional unit ball and sphere respectively, so that $\mathcal{H}^{d}\left(\mathcal{B}_{T_{q_{0}}M}(0,t_{0})\right)=\omega_{d}t_{0}^{d}$
	and $\mathcal{H}^{d-1}\left(\partial\mathcal{B}_{T_{q_{0}}M}(0,t_{0})\right)=\sigma_{d-1}t_{0}^{d-1}$.
	In view of deriving a lower bound on $\mathcal{H}^{d-1}\left(S_{1}\right)$, consider $u_{0}:=
	{t_{0}\gamma_{0}'(0)}
	\in S_{1}$. 
Since $\tau_{S_{1}}=t_{0}$
	and $\exp_{u_{0}}^{-1}(S_{1})\subset\mathcal{B}_{T_{u_{0}}S_{1}}\left(0,\tau_{\min}^{-\frac{1}{2}}t_{0}^{\frac{3}{2}}\right)$,
applying Proposition \ref{thm:geometry:injectivity_radius}
	(v) yields
	\begin{align*}
	\mathcal{H}^{d-1}\left(S_{1}\right) & \geq\left(1-\frac{t_{0}}{6\tau_{\min}}\right)^{d-1}\mathcal{H}^{d-1}\left(\mathcal{B}_{T_{u_{0}}S_{1}}\left(0,\tau_{\min}^{-\frac{1}{2}}t_{0}^{\frac{3}{2}}\right)\right)\\
	& \geq\left(\frac{59}{60}\right)^{d-1}\omega_{d-1}\tau_{\min}^{-\frac{d-1}{2}}t_{0}^{\frac{3d-3}{2}},
	\end{align*}
	and hence 
	\begin{align*}
	\mathcal{H}^{d-1}\left(\exp_{q_{0}}^{-1}(B_{1})\right) & =\frac{\mathcal{H}^{d}\left(\mathcal{B}_{T_{q_{0}}M}(0,t_{0})\right)\mathcal{H}^{d-1}\left(S_{1}\right)}{\mathcal{H}^{d-1}\left(\partial\mathcal{B}_{T_{q_{0}}M}(0,t_{0})\right)}\\
	& \geq \left(\frac{59}{60}\right)^{d-1}
	\frac{\omega_{d-1}}{d}\tau_{\min}^{-\frac{d-1}{2}}t_{0}^{\frac{3d-1}{2}}.
	\end{align*}
Finally, since $\exp_{q_{0}}^{-1}(B_{1})\subset\mathcal{B}_{T_{q_{0}}M}(q_{0},\frac{\tau_{M}}{10})$,
	Proposition \ref{thm:geometry:injectivity_radius} (v) yields
	\[
	\mathcal{H}^{d}\left(B_{1}\right)\geq\left(\frac{599}{600}\right)^{d}\mathcal{H}^{d}\left(\exp_{q_{0}}^{-1}(B_{1})\right)\geq\left(\frac{35341}{36000}\right)^{d}\frac{1}{d}\tau_{\min}^{-\frac{d-1}{2}}t_{0}^{\frac{3d-1}{2}},
	\]
	and hence,
	\[
	Q(B_{1})\geq\left(\frac{35341}{36000}\right)^{d}\frac{f_{\min}}{d}\tau_{\min}^{-\frac{d-1}{2}}t_{0}^{\frac{3d-1}{2}}\geq C_{\tau_{\min},d,L,f_{\min}}s^{\frac{3d-1}{2}}.
	\]
	By symmetry, the same bound holds for $Q(B_{2})$. 
	Applying these bounds to \eqref{eq:estimator_expansionbound} gives 
	\begin{align*}
	\mathbb{P}\left(\left|\frac{1}{\tau_{M}}-\frac{1}{\hat{\tau}(\X_n)}\right|>s\right) & \leq2\left(1-C_{\tau_{\min},d,L,f_{\min}}s^{\frac{3d-1}{2}}\right)^{n}\\
	& \leq2\exp\left(-C_{\tau_{\min},d,L,f_{\min}}ns^{\frac{3d-1}{2}}\right).
	\end{align*}
As for the proof of Proposition \ref{thm:estimator.global_risk}, the result then follows by integration.
\end{proof}

\section{Minimax Lower Bounds}

\subsection{Stability of the Model With Respect to Diffeomorphisms}
\label{sec:appendix:stability_diffeo}
To prove Proposition \ref{thm:diffeomorphism_stability}, we will use the following result stating that the reach is a stable quantity with respect to $\mathcal{C}^2$-perturbations.
\begin{lem}[Theorem 4.19 in \cite{Federer1959}]
\label{reach_stability}
Let $A \subset \mathbb{R}^D$ with $\tau_A \geq \tau_{min} > 0$ and $ \Phi  : \mathbb{R}^D \longrightarrow \mathbb{R}^D$ is a $\mathcal{C}^1$-diffeomorphism such that $ \Phi $,$ \Phi ^{-1}$, and $d  \Phi $ are Lipschitz with Lipschitz constants $K$,$N$ and $R$ respectively, then
	\[ \tau_{\Phi(A)} \geq \dfrac{\tau_{min}}{(K + R\tau_{min})N^2}. \]
\end{lem}

\begin{proof}[Proof of Proposition \ref{thm:diffeomorphism_stability}]
Let $M' = \Phi\left(M\right)$ be the image of $M$ by the mapping $\Phi$. Since $\Phi$ is a global diffeomorphism, $M'$ is a closed submanifold of dimension one. Moreover, $\Phi$ is $\norm{d \Phi}_{op} \leq (1 + \norm{ d \Phi - I_D }_{op})$-Lipschitz, $\Phi^{-1}$ is $\norm{d \Phi^{-1}}_{op} \leq ({1-\norm{ d \Phi - I_D }_{op}})^{-1}$-Lipschitz, and $d \Phi$ is $\norm{d^2 \Phi }_{op}$-Lipschitz. From Lemma \ref{reach_stability},
\begin{align*}
\tau_{M'} 
&\geq \frac{\tau_{min}(1 - \norm{ d \Phi - I_D }_{op})^2}{ \norm{d^2 \Phi }_{op} \tau_{min} + (1 + \norm{ d \Phi - I_D }_{op})} \geq \tau_{min}/2,
\end{align*}
where we used that $\norm{d^2 \Phi }_{op} \tau_{min} \leq 1/2$ and $\norm{ d \Phi - I_D }_{op} \leq 1/10$. All that remains to be proved now is the bound on the third order derivative of the geodesics of $M'$.
We denote by $\gamma$ and $\tilde{\gamma}$ the geodesics of $M$ and $M'$ respectively.

Let $p' = \Phi(p) \in M'$ and $v' = d_p \Phi . v \in T_{p'} M'$ be fixed.
Since $M \in \mathcal{M}^{d,D}_{\tau_{min},L}$ is a compact $\mathcal{C}^3$-submanifold with geodesics $\norm{\gamma'''(0)}\leq L$, $M$ can be parametrized locally by a $\mathcal{C}^3$ bijective map $\Psi_p: \mathcal{B}_{\R^d}(0,\varepsilon) \rightarrow M$ with $\Psi_p(0) = p$.
For a smooth curve $\gamma$ on $M$ nearby $p$, we let $c = (c_1,\ldots,c_d)^t$ denote its lift in the coordinates $\mathbf{x} = \Psi_p^{-1}$, that is $\gamma(t) = \Psi_p \circ c(t)$.
$\gamma = \gamma_{p,v}$ is the geodesic of $M$ with initial conditions $p$ and $v$ if and only if $c$ satisfies the geodesic equations (see \cite[p.62]{DoCarmo92}).
That is, the second order ordinary differential equation
\begin{align}\label{geodesic_equation}
\begin{cases}
c''_\ell(t)
+ 
\left\langle 
\Gamma^\ell\left( c(t) \right) \cdot c'(t)
, 
c'(t)
\right\rangle 
=
0,
\hspace{2em} (1 \leq \ell \leq d)
\\
c(0) = 0 \text{ and }  c'(0) = d_p \mathbf{x} . v,
\end{cases}
\end{align}
where $\Gamma^\ell = \bigl( \Gamma^\ell_{i,j} \bigr)_{1 \leq i,j \leq d}$ are the Christoffel symbols of the $\mathcal{C}^3$ chart $\mathbf{x}$, which depends only on $\mathbf{x}$ and its differentials of order $1$ and $2$.
By construction, $M'$ is parametrized locally by ${\Psi}'_{p'} = \Phi \circ \Psi_{p}$ yielding local coordinates $\mathbf{y} = {\Psi'}_{p'}^{-1} = \Psi_p^{-1} \circ \Phi^{-1}$ nearby $p' \in M'$.
Writing $\tilde{\Gamma}^\ell$ for the Christoffel's symbols of $M'$, $\tilde{\gamma}$ is a geodesic of $M'$ at $p'$ if its lift $\tilde{c} = \Psi_{p'}'^{-1} (\tilde{\gamma})$ satisfies \eqref{geodesic_equation} with $\Gamma^\ell$ replaced by $\tilde{\Gamma}^\ell$, and initial conditions $\tilde{c}(0) = c$ and $\tilde{c}'(0) = d_{p'} \mathbf{y}. v' = d_p \mathbf{x}.v$. From chain rule, the $\tilde{\Gamma}^\ell$'s depend on $\Gamma$, $d \Phi$, and $d^2 \Phi$.

Write $c'''(0) - \tilde{c}'''(0)$ by differentiating \eqref{geodesic_equation}: since $c(0) = \tilde{c}(0) = 0$ and $c''(0) = \tilde{c}''(0)$, we get that for $\norm{I_D - d \Phi}_{op}$, $\norm{d^2 \Phi}_{op}$ and $\norm{d^3 \Phi}_{op}$ small enough, $\norm{c'''(0) - \tilde{c}'''(0)}$ can be made arbitrarily small. 
In particular, $\tilde{\gamma}'''(0)$ gets arbitrarily close to $\gamma'''(0)$, so that $\norm{\tilde{\gamma}'''(0)} \leq \norm{\gamma'''(0)} + L \leq 2L$, which concludes the proof.
\end{proof}

\subsection{Lemmas on the Total Variation Distance}
\label{sec:appendix:TV_lemmas}
Prior to any actual construction, we show the following straightforward lemma bounding the total variation between uniform distribution on manifolds that are perturbations of each other.
For $M \subset \R^D$, write $\lambda_M = {\mathbbm{1}_M} \mathcal{H}^d / {\mathcal{H}^d(M)}$ for the uniform probability distribution on $M$.
\begin{lem}\label{total_variation_between_uniforms}
Let $M\subset \R^D$ be a compact $d$-dimensional submanifold and ${B} \subset \R^D$ be a Borel set. Let $\Phi : \R^D \rightarrow \R^D$ be a global diffeomorphism such that $\Phi_{|B^c}$ is the identity map and $\norm{d \Phi - I_D}_{op} \leq 2^{1/d}-1$. Then ${{\mathcal{H}^d}(\Phi(M))}\leq 2{{\mathcal{H}^d}(M)} $ and $ TV\left(\lambda_M,\lambda_{\Phi\left(M\right)}\right) \leq 12 \lambda_M\left(B\right)$.
\end{lem}

\begin{proof}[Proof of Lemma \ref{total_variation_between_uniforms}]
Since $\Phi$ is $(1+ \norm{d \Phi - I_D}_{op})$-Lipschitz, \cite[Lemma 7]{Arias13} asserts that 
\[
{\mathcal{H}^d}\left(\Phi(M \cap B)\right)\leq (1+ \norm{d \Phi - I_D}_{op})^d {\mathcal{H}^d}(M\cap B) \leq 2 {\mathcal{H}^d}(M\cap B).
\]
Therefore,
\begin{align*}
{\mathcal{H}^d}\left( \Phi(M) \right) - {\mathcal{H}^d}(M) 
&= {\mathcal{H}^d}\left(\Phi(M \cap B)\right) - {\mathcal{H}^d}\left(M \cap B\right) \\
&\leq {\mathcal{H}^d}(M \cap B) \leq {\mathcal{H}^d}(M).
\end{align*}
Now, writing $\triangle$ for the symmetric difference of sets, we have $M \triangle \Phi(M) = (B\cap M)\triangle(B\cap \Phi(M)) \subset (B\cap M)\cup(B\cap \Phi(M))$. Therefore, \cite[Lemma 7]{Arias13} yields,
\begin{align*}
TV\left(\lambda_M,\lambda_{\Phi\left(M\right)}\right)
&\leq 4 \frac{{\mathcal{H}^d}\left(M \triangle \Phi(M) \right) }{{\mathcal{H}^d}(M \cup \Phi(M))}\\
&\leq 4\frac{{\mathcal{H}^d}\left(M \cap B \right) + {\mathcal{H}^d}\left(\Phi(M) \cap B \right) }{{\mathcal{H}^d}(M)} \\
&= 4\frac{{\mathcal{H}^d}\left(M \cap B \right) + {\mathcal{H}^d}\left(\Phi(M\cap B)\right) }{{\mathcal{H}^d}(M)} \\
&\leq 12 \frac{{\mathcal{H}^d}(M\cap B)}{{\mathcal{H}^d}(M)} = 12 \lambda_M(B).
\end{align*}
\end{proof}

Let us now tackle the proof of Lemma \ref{thm:total_variation_identiy}. For this, we will need the following elementary differential geometry results Lemma  \ref{lem:nonsingular_level_set} and Corollary \ref{thm:non_parallel_intersections_have_measure_zero}.

\begin{lem}\label{lem:nonsingular_level_set}
Let $g: \R^d \rightarrow \R^k$ be $\mathcal{C}^1$ and $x\in \R^d$ be such that $g(x) = 0$ and $d_x g \neq 0$. Then there exists $r>0$ such that $\mathcal{H}^d\left( g^{-1}(0) \cap \mathcal{B}(x,r) \right) = 0$.
\end{lem}

\begin{proof}[Proof of Lemma \ref{lem:nonsingular_level_set}]
Let us prove that for $r>0$ small enough, the intersection $g^{-1}(0) \cap \mathcal{B}(x,r)$ is contained in a submanifold of codimension one of $\R^d$. Writing $g=(g_1,\ldots,g_k)$, assume without loss of generality that $\partial_{x_1} g_1 \neq 0$. Since $g_1 : \R^d \rightarrow \R$ is non-singular at $x$, the implicit function theorem asserts that $g_1^{-1}(0)$ is a submanifold of dimension $d-1$ of $\R^d$ in a neighborhood of $x \in \R^d$. Therefore, for $r>0$ small enough, $g_1^{-1}(0)\cap \mathcal{B}(x,r)$ has $d$-dimensional Hausdorff measure zero. The result hence follows, noticing that $g^{-1}(0) \subset g_1^{-1}(0)$.
\end{proof}

\begin{cor}\label{thm:non_parallel_intersections_have_measure_zero}
Let $M,M' \subset \R^D$ be two compact $d$-dimensional submanifolds, and $x\in M\cap M'$. If $T_x M \neq T_x M'$, there exists $r>0$ such that $A = M\cap M' \cap \mathcal{B}(x,r)$ satisfies $\lambda_M(A) = \lambda_{M'}(A) = 0$.
\end{cor}

\begin{proof}[Proof of Corollary \ref{thm:non_parallel_intersections_have_measure_zero}]
Writing $k= D-d$, we see that up to ambient diffeomorphism --- which preserves the nullity of measure --- we can assume that locally around $x$, $M'$ coincides with $\R^d \times \left\{0 \right\}^{k}$ and that $M$ is the graph of a $\mathcal{C^1}$ function $g: \mathcal{B}_{\R^d}(0,r') \rightarrow \R^k$ for $r'>0$ small enough. The assumption $T_x M \neq T_x M'$ translates to $d_0 g \neq 0$, and the previous transformation maps smoothly $M\cap M' \cap \mathcal{B}(x,r'')$ to $g^{-1}(0)\cap \mathcal{B}(0,r')$ for $r''>0$ small enough.
We conclude by applying Lemma \ref{lem:nonsingular_level_set}.
\end{proof}

We are now in position to prove Lemma \ref{thm:total_variation_identiy}.

\begin{proof}[Proof of Lemma \ref{thm:total_variation_identiy}]
Notice that $Q$ and $Q'$ are dominated by the measure $\mu = \mathbbm{1}_{M \cup M'} \mathcal{H}^d$, with $d Q(x) = f(x) d \mu(x)$ and $d Q'(x) = f'(x) d \mu(x)$, where $f,f': \R^D \rightarrow \R_+$ have support $M$ and $M'$ respectively.
On the other hand, $P$ and $P'$ are dominated by $\nu(d x \, d T) = \delta_{\left\{T_x M, T_x M'\right\} }\left(d T\right) \mu\left(d x\right)$ with respective densities $\bar{f}(x,T) = \mathbbm{1}_{T= T_x M} f(x)$ and $\bar{f}'(x,T) = \mathbbm{1}_{T= T_x M'}f'(x)$, where we set arbitrarily $T_x M = T_0$ for $x\notin M$, and $T_x M' = T_0$ for $x\notin M'$.
Recalling that $f$ vanishes outside $M$ and $f'$ outside $M'$,
\begin{align*}
TV&(P,P')
\\
&= 
\frac{1}{2}\int_{\mathbb{R}^D \times \mathbb{G}^{d,D}} |\bar{f}-\bar{f}'| d \nu
\\
&=
\frac{1}{2}\int_{\mathbb{R}^D}   \mathbbm{1}_{T_x M = T_x M'}| f(x) - f'(x)| + \mathbbm{1}_{T_x M \neq T_x M'}( f(x) + f'(x))   \mathcal{H}^d(d x).
\end{align*}
From Corollary \ref{thm:non_parallel_intersections_have_measure_zero} and a straightforward compactness argument, we derive that 
\[
\mathcal{H}^d\left( M \cap M' \cap \left\{x | T_x M \neq T_x M' \right\} \right) = 0
.
\] As a consequence, the above integral expression becomes
\begin{align*}
TV(P,P')
&=
\frac{1}{2}\int_{\mathbb{R}^D} |f-f'| d \mathcal{H}^d
=TV(Q,Q'),
\end{align*}
which concludes the proof.
\end{proof}

\subsection{Construction of the Hypotheses}
\label{sec:appendix:construction}
This section is devoted to the construction of hypotheses that will be used in Le Cam's lemma (Lemma \ref{lecam_lemma}), to derive Proposition \ref{thm:minimax_nonconsistency} and Theorem \ref{thm:lower_bound_thm}.

\begin{lem}\label{thm:two_hypotheses_construction_generic}
Let $R, \ell, \eta>0$ be such that $\ell \leq \frac{R}{2}\wedge \left(2^{1/d}-1\right)$ and $\eta \leq \frac{\ell^2}{2R}$.
Then there exists a $d$-dimensional sphere of radius $R$ that we call $M$, such that $M \in \mathcal{M}^{d,D}_{R,\frac{1}{R^2}}$ and a global $\mathcal{C}^\infty$-diffeomorphism $\Phi: \R^D \rightarrow \R^D$ such that,
\begin{align*}
\norm{d \Phi - I_D}_{op} \leq \frac{3\eta}{\ell}
, \hspace{1em}
\norm{d^2 \Phi}_{op} \leq \frac{23 \eta}{\ell^2}
, \hspace{1em}
\norm{d^3 \Phi}_{op} \leq \frac{573 \eta}{\ell^3},
\end{align*}
and so that writing $M' = \Phi(M)$, we have
$\mathcal{H}^d(M') \leq 2 \mathcal{H}^d(M) = 2 \sigma_d R^d$,
\begin{align*}
\left| \frac{1}{\tau_{{M}}} - \frac{1}{\tau_{{M}'}}\right| \geq \frac{\eta}{\ell^2},
\hspace{1em} \text{ and } \hspace{1em}
TV(\lambda_M,\lambda_{M'}) \leq 12\left(\frac{\ell}{R}\right)^d
.
\end{align*}
\end{lem}

\begin{proof}[Proof of Lemma \ref{thm:two_hypotheses_construction_generic}]
Let ${M} \subset \R^{d+1} \times \{0\}^{D-d-1} \subset \R^D$ be the sphere of radius $R$ with center $(0,-R,0,\ldots,0)$.
The reach of $M$ is $\tau_{M} = R$, and its arc-length parametrized  geodesics are arcs of great circles, which have third derivatives of constant norm $\norm{\gamma'''(t)} = \frac{1}{R^2}$. Hence we see that $M \in \mathcal{M}^{d,D}_{R,\frac{1}{R^2}}$.
Let $\phi : \R^D \rightarrow \R_+$ be the map defined by $\phi(x) = \exp \bigl( \frac{\norm{x}^2}{\norm{x}^2-1} \bigr) \mathbbm{1}_{\norm{x}^2<1}$. 
$\phi$ is a symmetric $\mathcal{C}^\infty$ map with support equal to $\mathcal{B}(0,1)$ and elementary real analysis yields $\phi(0) = 1$, $\norm{d \phi}_{op} \leq 3$, $\norm{d^2 \phi}_{op} \leq 23$ and $\norm{d^3 \phi}_{op} \leq 573$. Let $\Phi : \R^D \rightarrow \R^D$ be defined by
\begin{align*}
\Phi(x) = x + \eta \phi\left( x/\ell \right)\cdot v,
\end{align*}
where $v = (0,1,0,\ldots,0)$ is the unit vertical vector. 
$\Phi$ is the identity map on $\mathcal{B}\left(0,\ell\right)^c$, and in $\mathcal{B}\left(0,\ell\right)$, $\Phi$ translates points on the vertical axis with a magnitude modulated by the weight function $\phi(x/\ell)$. 
From chain rule, $\norm{d \Phi - I_D}_{op} = \eta \norm{d \phi}_{\infty}/ \ell \leq 3\eta / \ell <1$. 
Therefore, $d_x \Phi$ is invertible for all $x\in \R^D$, so that $\Phi$ is a local $\mathcal{C}^\infty$-diffeomorphism according to the local inverse function theorem. 
Moreover, $\norm{\Phi(x)} \rightarrow \infty$ as $\norm{x} \rightarrow \infty$, so that $ \Phi $ is a global $\mathcal{C^\infty}$-diffeomorphism by Hadamard-Cacciopoli theorem \cite{DeMarco1994}. 
Similarly, from bounds on differentials of $\phi$ we get
\[
\norm{d^2 \Phi}_{op} \leq 23\frac{\eta}{\ell^2}
\hspace{1em}
\text{ and }
\hspace{1em}
\norm{d^3 \Phi}_{op} \leq 573 \frac{\eta}{\ell^3}.
\]

Let us now write $M' = \Phi\left(M \right)$ for the image of $M$ by the map $\Phi$ (see Figure~\ref{figure_lower_bound_no_boundary}). 
Denote by $\left(Oy\right)$ the vertical axis $\mathrm{span}(v)$, and notice that since $\phi$ is symmetric, $M'$ is symmetric with respect to the vertical axis $\left(Oy\right)$.
We now bound from above the reach $\tau_{M'}$ of $M'$ by showing that the point $x_0 = \left(0, \frac{R+\eta/2}{1+\frac{\ell^2}{2R\eta}} , 0,\ldots,0\right)$ belongs to its medial axis $Med(M')$ (see \eqref{eq:medialaxis}).
\begin{figure}
\centering
\includegraphics[clip = true, trim = 0 140 0 0mm,width=0.6\textwidth]{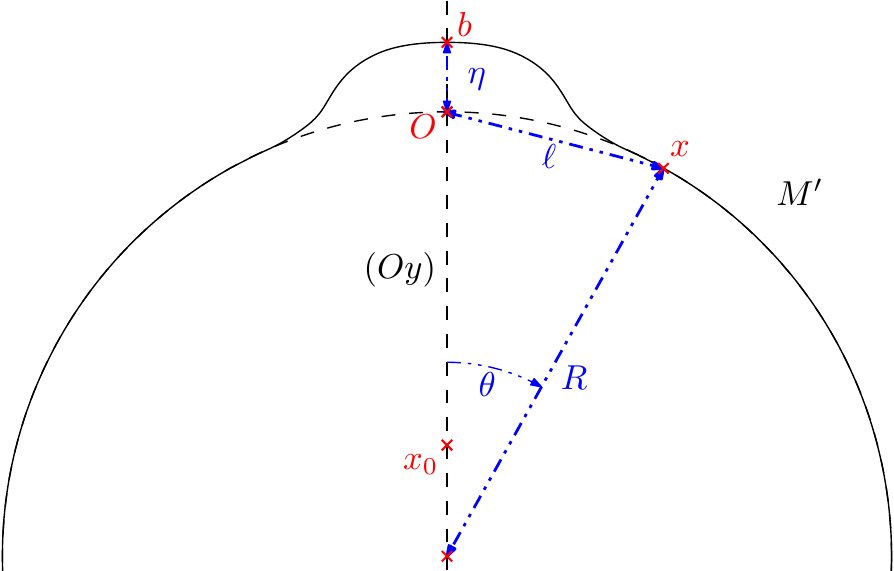}
\caption{The bumped sphere $M'$ of Lemma \ref{thm:two_hypotheses_construction_generic}.}
\label{figure_lower_bound_no_boundary}
\end{figure}
For this, write 
\begin{align*}
b = (0,\eta, 0,\ldots,0)
,
\hspace{1em}
b'=(0,-2R, 0,\ldots,0)
, 
\end{align*}
together with $\theta = \arccos (1 - \ell^2/(2R^2))$, and
\begin{align*}
x = (R\sin \theta, R \cos \theta -R, 0,\ldots,0).
\end{align*}
By construction, $b,b'$ and $x$ belong to $M'$. 
One easily checks that $\norm{x_0-x} < \norm{x_0-b}$ and $\norm{x_0-x} < \norm{x_0-b'}$, so that neither $b$ nor $b'$ is the nearest neighbor of $x_0$ on $M'$. 
But $x_0 \in \left(Oy\right)$ which is an axis of symmetry of $M'$, and $\left(Oy \right) \cap M' = \left\{b,b'\right\}$. 
As a consequence, $x_0$ has strictly more than one nearest neighbor on $M'$. That is, $x_0$ belongs to the medial axis $Med(M')$ of $M'$. 
Therefore,
\begin{align*}
\frac{1}{\tau_{M'}}
\geq
\frac{1}{d\left(x_0,M'\right)}
&\geq \frac{1}{\norm{x_0-x}} \\
&\geq 
\frac{1}{R\left| 1 - \frac{\ell^2}{2R^2}- \frac{1+\frac{\eta}{2R}}{1+\frac{\ell^2}{2R\eta}}  \right| } \\
&\geq \frac{1}{R \left(1 - \frac{1+\frac{\eta}{2R}}{1+\frac{\ell^2}{2R\eta}}\right) }
\geq \frac{1}{R}\left(1 + \frac{1+\frac{\eta}{2R}}{1+\frac{\ell^2}{2R\eta}}\right)
\geq \frac{1}{R} + \frac{\eta}{\ell^2},
\end{align*}
which yields the bound $\left| \frac{1}{\tau_{M}} - \frac{1}{\tau_{M'}}\right| = \left| \frac{1}{R} - \frac{1}{\tau_{M'}}\right|\geq \frac{\eta}{\ell^2}$.

Finally, since $M' = \Phi(M)$ with $\norm{d \Phi - I_D}_{op} \leq 2^{1/d}-1$ with $\Phi_{|\mathcal{B}(0,\ell)^c}$ coinciding with the identity map, Lemma \ref{total_variation_between_uniforms} yields 
$\mathcal{H}^d(M') \leq 2 \mathcal{H}^d(M) = 2 \sigma_d R^d$ and
\begin{align*}
TV\left(\lambda_M, \lambda_{M'}\right)
&\leq 
12 \lambda_M\left(\mathcal{B}(0,\ell)\right)
\\
&= 
12
\frac{
	\mathcal{H}^d\left(\mathcal{B}_{\mathcal{S}^d}\left(0,2 \arcsin\left(\frac{\ell}{2R}\right)\right)\right)
	}
	{
	\mathcal{H}^d\left(\mathcal{S}^d\right)
	} 
\\
&\leq 
12 \left( \frac{\ell}{R}\right)^d
,
\end{align*}
which concludes the proof.
\end{proof}

\begin{proof}[Proof of Proposition \ref{thm:two_hypotheses_construction}]

Apply  Lemma \ref{thm:two_hypotheses_construction_generic} with $R = 2\tau_{min}$. Then the sphere $M$ of radius $2\tau_{min}$ belongs to $\mathcal{M}^{d,D}_{2\tau_{min},1/(4\tau_{min}^2)}$. Furthermore, taking $\eta = c_{d} {\ell^3}/\tau_{min}^2$ for $c_{d} >0$ and $\ell > 0$ small enough, Proposition \ref{thm:diffeomorphism_stability} (applied to the unit sphere, yielding $c_d$, and reasoning by homogeneity for the sphere of radius $2\tau_{min}$) asserts that $M' = \Phi(M)$ belongs to $\mathcal{M}^{d,D}_{\tau_{min},1/(2\tau_{min}^2)} \subset \mathcal{M}^{d,D}_{\tau_{min},L}$, since $L \geq 1/(2\tau_{min}^2)$.
Moreover, 
\begin{align*}
\mathcal{H}^d(M')^{-1} \wedge \mathcal{H}^d(M)^{-1} 
\geq 
\bigl(
2^{d+1} \sigma_d \tau_{min}^d
\bigr)^{-1}
\geq
f_{min},
\end{align*}
so that $\lambda_M,\lambda_{M'} \in \mathcal{Q}^{d,D}_{\tau_{min},L,f_{min}}$, which gives the result.
\end{proof}

Let us now prove the minimax inconsistency of the reach estimation for $L = \infty$, using the same technique as above.

\begin{proof}[Proof of Proposition \ref{thm:minimax_nonconsistency}]
Let $M$ and $M'$ be given by Lemma \ref{thm:two_hypotheses_construction_generic} with $\ell \leq \frac{R}{2} \wedge (2^{1/d}-1)$,  $\eta = \ell^2/(23R)$ and $R = 2\tau_{min}$. We have $\norm{d \Phi - I_D}_{op} \leq 3\eta/\ell \leq 0.1$ and $\norm{d^2 \Phi}_{op} \leq 23 \eta/\ell^2 \leq 1/(2 \tau_{min})$. Since $\tau_M \geq 2 \tau_{min}$, Lemma \ref{reach_stability} yields
\begin{align*}
\tau_{M'} 
&\geq \frac{\tau_{M}(1 - \norm{ d \Phi - I_D }_{op})^2}{ \norm{d^2 \Phi }_{op} \tau_{M} + (1 + \norm{ d \Phi - I_D }_{op})} 
\geq \tau_{min}.
\end{align*}
As a consequence, $M$ and $M'$ belong to $\mathcal{M}^{d,D}_{\tau_{min}, L = \infty}$. 
Furthermore, since we have $f_{min} \leq \left({2^{d+1}\tau_{min}^d \sigma_d}\right)^{-1} \leq $ $\mathcal{H}^d(M)^{-1} \wedge \mathcal{H}^d(M')^{-1}$, we see that the uniform distributions $\lambda_M,\lambda_{M'}$ belong to $\mathcal{Q}^{d,D}_{\tau_{min},L = \infty,f_{min}}$.
Let now $P,P'$ denote the distributions of $\mathcal{P}^{d,D}_{\tau_{min},L = \infty,f_{min}}$ associated to $\lambda_M,\lambda_{M'}$ (Definition \ref{def:model_with_known_tangent_spaces}).
Lemma \ref{thm:total_variation_identiy} asserts that $TV(P,P') = TV(\lambda_M,\lambda_{M'})$.
Applying Lemma \ref{lecam_lemma} to $P,P'$, we get that for all $n\geq 1$, for $\ell$ small enough,
\begin{align*}
\inf_{\hat{\tau}_n} \sup_{P \in \mathcal{P}^{d,D}_{\tau_{min},L=\infty, f_{min}}} \E_{P^n} \left| \frac{1}{\tau_P} - \frac{1}{\hat{\tau}_n} \right|^p
&\geq
\frac{1}{2^{p}}\left|\frac{1}{\tau_M}-\frac{1}{\tau_{M'}}\right|^p \left(1-TV(P,P')\right)^{n} 
\\
&\geq
\frac{1}{2^p}
\left(
\frac{\eta}{\ell^2}
\right)^p
\left(
1
-
12\left(\frac{\ell}{2\tau_{min}}\right)^d
\right)^n
\\
&=
\frac{1}{2^p}
\left(
\frac{1}{46 \tau_{min}}
\right)^p
\left(
1
-
12\left(\frac{\ell}{2\tau_{min}}\right)^d
\right)^n
.
\end{align*}
Sending $\ell \rightarrow 0$ with $n\geq 1$ fixed yields the announced result.

\end{proof}

\section{Stability with Respect to Tangent Spaces}
\label{sec:appendix:stability_tangent}
\begin{proof}[Proof of Proposition \ref{thm:tangent_stability}]
To get the bound on the difference of suprema, we show the (stronger) pointwise bound. Indeed, for all $x,y \in \X$ with $x\neq y$,
\begin{align*}
\left| \frac{2d(y-x,T_x)}{\norm{y-x}^2} - \frac{2d(y-x,\tilde{T}_x)}{\norm{y-x}^2}\right|
&\leq \frac{2\Vert{\pi_{T_x}(y-x) - \pi_{\tilde{T}_x}(y-x)}\Vert}{\norm{y-x}^2}  \\
&\leq \frac{2 \Vert {\pi_{T_x} - \pi_{\tilde{T}_x}}\Vert _\mathrm{op} }{\norm{y-x}} 
\leq \frac{2 \sin \theta}{\delta}.
\end{align*}
\end{proof}

\bibliographystyle{imsart-number}
\bibliography{biblioReach}
\end{document}